\documentclass[a4paper, 10pt]{article}

\usepackage{inputenc, amsmath, amsthm, amssymb}
\usepackage{latexsym}
\usepackage{geometry}
\usepackage{mathrsfs}
\usepackage{ascii}
\usepackage{makeidx}
\makeindex

\usepackage{fancyhdr, color, multicol, graphicx}

\DeclareMathAlphabet{\mathpzc}{OT1}{pzc}{m}{it}

\numberwithin{equation}{section}

\DeclareMathAlphabet{\mathpzc}{OT1}{pzc}{m}{it}

\fancypagestyle{standard}
{\fancyhead[RE]{\nouppercase{ \leftmark}}
\fancyhead[LO]{\nouppercase{ \rightmark}}
\fancyhead[RO,LE]{\thepage}
\fancyfoot{}}

\fancypagestyle{plain}
{\fancyhead{}\fancyfoot[C]{\thepage}}

\newtheorem{thm}{Theorem}[section]
\newtheorem{lem}[thm]{Lemma}
\newtheorem{prop}[thm]{Proposition}

\theoremstyle{definition}

\newtheorem{rem}[thm]{Remark}

\newcommand{\eps}{\ensuremath{\varepsilon}}

\begin{document}
\begin{center}
\LARGE{\textbf{A Combination of Downward Continuation and Local Approximation for Harmonic Potentials}}
\\[3ex]\normalsize C. Gerhards\footnote{Geomathematics Group, University of Kaiserslautern, PO Box 3049, 67663 Kaiserslautern
\\e-mail: gerhards.christian@gmail.com}
\\[5ex]
\end{center}

\textbf{Abstract.}
This paper presents a method for the approximation of harmonic potentials that combines downward continuation of globally available data on a sphere $\Omega_R$ of radius $R$ (e.g., a satellite's orbit) with locally available data in a subregion $\Gamma_r$ of the sphere $\Omega_r$ of radius $r<R$ (e.g., the spherical Earth's surface). The approximation is based on a two-step algorithm motivated by spherical multiscale expansions: First, a convolution with a scaling kernel $\Phi_N$ deals with the downward continuation from $\Omega_R$ to $\Omega_r$, while in a second step, the result is locally refined by a convolution on $\Omega_r$ with a wavelet kernel $\tilde{\Psi}_N$. The kernels $\Phi_N$ and $\tilde{\Psi}_N$ are optimized in such a way that the former behaves well for the downward continuation while the latter shows a good localization in $\Gamma_r$. 
The concept is indicated for scalar as well as vector potentials. 
\\

\\


\section{Introduction}

Recent satellite missions monitoring the Earth's gravity and magnetic field supply a large amount of data with a fairly good global coverage. They are complemented by local/regional measurements at or near the Earth's surface. While satellite data is well-suited for the reconstruction of large-scale structures, it fails for spatially localized features (due to the involved downward continuation). The opposite is true for locally/regionally available ground data. It is well-suited to capture local phenomena but fails for global trends. Therefore, in order to obtain high-resolution gravitational models, such as EGM2008 (cf. \cite{pavlis}), or geomagnetic models, such as NGDC-720\footnote{http://geomag.org/models/ngdc720.html}, it becomes necessary to combine both types of data. The upcoming Swarm satellite mission, e.g., aims at reducing the (spectral) gap between satellite data and local/regional data at or near the Earth's surface (cf. \cite{swarm}) by supplying improved data from a constellation of three satellites. Making use of this improved situation requires methods that address the different properties of satellite and ground data. 

In order to deal with local/regional data sets, various types of localizing spherical basis functions have been developed during the last years and decades. Among them are spherical splines (e.g., \cite{freeden81}, \cite{shure82}), spherical cap harmonics (e.g., \cite{haines85}, \cite{thebault}), and Slepian functions (e.g., \cite{plattner13}, \cite{plattner13b}, \cite{simons06}, \cite{simons10}). Spherical multiscale methods go a bit further and allow a scale-dependent adaptation of scaling and wavelet kernels (see, e.g., \cite{dahlke}, \cite{freewind}, \cite{hol}, and \cite{sweldens} for the early development). They are particularly well-suited to combine global and local/regional data sets of different resolution and have been applied intensively to problems in geomagnetism and gravity field modeling, e.g., 
in \cite{bayer01}, \cite{chambodut}, \cite{freeger}, \cite{freeschrei06}, \cite{freewind}, \cite{ger12}, \cite{ger13}, \cite{hol03}, \cite{klees07}, \cite{mai05}, \cite{mayer06}, and \cite{michel01}. Matching pursuits as described, e.g., in \cite{mallat} have been adapted more recently to meet the requirements of geoscientific problems (cf. \cite{michel12}, \cite{fischer13}). Their dictionary structure allows the inclusion of a variety of global and spatially localizing basis functions, of which adequate functions are selected automatically dependent on the given data.

However, when combining satellite data on a sphere $\Omega_R$ and local/regional data on a sphere $\Omega_r$ of radius $r<R$, not only methods that are able to deal with the local/regional aspect become necessary but also those that deal with the ill-posedness of downward continuation of data on $\Omega_R$. Typically, those two problems are treated separately. (Spherical) downward continuation itself has been studied intensively, e.g., in \cite{bauer13}, \cite{freeden99}, \cite{freeden01}, \cite{schneider98}, \cite{pereverzev10}, \cite{naumov}, \cite{perev99} (for the more mathematical aspects) and \cite{cooper}, \cite{maimay03}, \cite{trompat}, \cite{tziavos} (for a stronger focus on the geophysical application). A particular approach to regularize downward continuation is given by multiscale methods (see, e.g., \cite{freeden99}, \cite{freeden01}, \cite{schneider98}, \cite{maimay03}, \cite{perev99} for the particular case of spherical geometries). 
Yet, it seems that no approach intrinsically combines the two problems, especially regarding that downward continuation is required for the data on $\Omega_R$ but not for the local/regional data on $\Omega_r$.

It is the goal of this paper, motivated by some of the previous multiscale methods, to introduce a two-step approximation reflecting such an intrinsic combination. More precisely, in the first step only data on $\Omega_R$ is used and downward continued by convolution with a scaling kernel $\Phi_N$. In the second step, the approximation is refined by convolving the local/regional data on $\Omega_r$ with a spatially localizing wavelet kernel $\tilde{\Psi}_N$. The connection of the two steps is given by the construction of the kernels $\Phi_N$,  $\tilde{\Psi}_N$: Both kernels are designed in such a way that they simultaneously minimize a functional that contains a penalty term for the downward continuation and a penalty term for spatial localization.  Thus, it is not the goal to first get a best possible approximation from satellite data only and then refine this approximation with local/regional data. It is rather to find a balance between the data on $\Omega_R$ and the data on $\Omega_r$ that in some sense 
leads to a best overall approximation.

\begin{figure}
\begin{center}
\scalebox{0.3}{\includegraphics{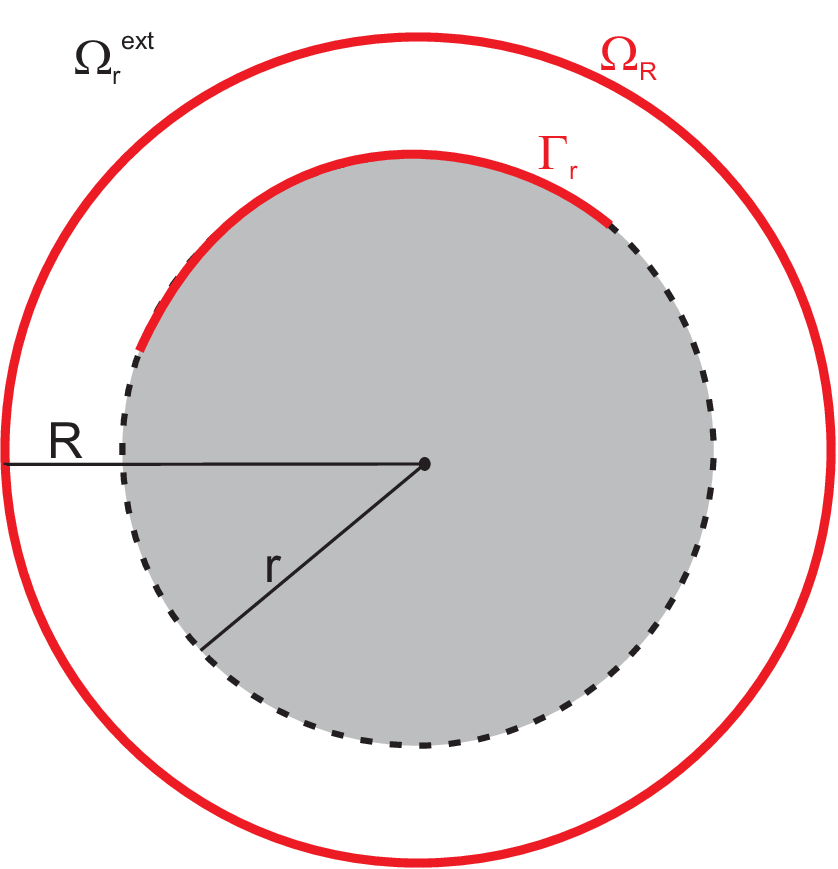}}
\end{center}
\caption{The given data situation.}\label{fig:datasit}
\end{figure}

\subsection{Brief Description of the Approach}

In the exterior of the Earth, the gravity and the crustal magnetic field can be described by a harmonic potential $U$. From satellite measurements we obtain data $F_1$ on a spherical orbit $\Omega_R=\{x\in\mathbb{R}^3:|x|=R\}$ and from ground or near-ground measurements data $F_2$ in a subregion $\Gamma_r$ of the spherical Earth surface $\Omega_r$ of radius $r<R$ (cf. Figure \ref{fig:datasit}). The problem to solve is
\begin{align}
\Delta U=0,\quad&\textnormal{ in }\Omega_r^{ext},\label{eqn:11}
\\U=F_1,\quad&\textnormal{ on }\Omega_R,\label{eqn:12}
\\U=F_2,\quad&\textnormal{ on }\Gamma_r,\label{eqn:13}
\end{align}
with $\Omega_r^{ext}=\{x\in\mathbb{R}^3:|x|>r\}$ denoting the space exterior to the sphere $\Omega_r$. Of interest to us is the restriction $U^+=U|_{\Gamma_r}$, i.e., the potential in the subregion $\Gamma_r$ of the Earth's surface. The knowledge of $F_1$ on $\Omega_R$ already supplies all information necessary to obtain $U^+$. However, since it is only available by measurements in discrete points on $\Omega_R$, possible noise and the involved downward continuation render $F_1$ only suitable to approximate the coarser structures of $U^+$. Additional measurements of $F_2$ in $\Gamma_r$ improve the situation.  

Throughout this paper, we use an approximation $U_N$ of $U^+$ of the form 
\begin{align}\label{eqn:firstapprox}
{U_N={T}_N[F_1]+{\tilde{W}}_N[F_2].}
\end{align}
It is motivated by spherical multiscale representations as introduced in \cite{schneider98} and \cite{freewind}: $T_N$ reflects a regularized version of the downward continuation operator, acting as a scaling transform on $\Omega_R$ with the convolution kernel 
\begin{align}\label{eqn:phin}
{\Phi_N(x,y)=\sum_{n=0}^N\sum_{k=1}^{2n+1}\Phi_N^\wedge(n)\frac{1}{r}Y_{n,k}\left(\xi\right)\frac{1}{R}Y_{n,k}\left(\eta\right).}
\end{align}
We frequently use $\xi$ and $\eta$ to abbreviate the unit vectors $\frac{x}{|x|}$ and $\frac{y}{|y|}$, respectively, and write $r=|x|$, $R=|y|$. Furthermore, $\{Y_{n,k}\}_{n=0,1,\ldots;k=1,\ldots,2n+1}$ denotes a set of orthonormal spherical harmonics of degree $n$ and order $k$. In order to refine the approximation with local data we use the operator $\tilde{W}_N$, which acts as a wavelet transform on $\Gamma_r$ with the convolution kernel
\begin{align}\label{eqn:psin}
{\tilde{\Psi}_N(x,y)=\sum_{n=0}^{\lfloor\kappa N\rfloor}\sum_{k=1}^{2n+1}\tilde{\Psi}_N^\wedge(n)\frac{1}{r}Y_{n,k}\left(\xi\right)\frac{1}{r}Y_{n,k}\left(\eta\right), }
\end{align}
where $\kappa>1$ is a fixed constant (reflecting the higher resolution desired for the refinement). The coefficients ${\Phi}_N^\wedge(n)$ and $\tilde{\Psi}_N^\wedge(n)$ are typically called 'symbols' of the corresponding kernels. They are coupled by the relation $\tilde{\Psi}_N^\wedge(n)=\tilde{\Phi}_N^\wedge(n)-{\Phi}_N^\wedge(n)\big(\frac{r}{R}\big)^n$ (see Section \ref{subsec:comb} for details), where $\tilde{\Phi}_N^\wedge(n)$ has been introduced as an auxiliary symbol. This coupling guarantees a smooth transition from the use of global satellite data on $\Omega_R$ to local data in $\Gamma_r$. The optimization of the kernels is done by simultaneously choosing symbols $\Phi_N^\wedge(n)$, $\tilde{\Phi}_N^\wedge(n)$ that minimize a functional $\mathcal{F}$ reflecting the desired properties. 

The general setting and notation as well as the choice of the functional $\mathcal{F}$ are described in Sections \ref{sec:sett} and \ref{sec:min}. Convergence results for the approximation are supplied in Section \ref{sec:theo} and numerical tests in Section \ref{sec:num}. In Section \ref{sec:vect}, we transfer the concept to a vectorial setting, where the gradient $\nabla U$ is approximated from vectorial data on $\Omega_R$ and $\Gamma_r$. This is of interest, e.g., for the crustal magnetic field where the actual sought-after quantity is the vectorial magnetic field $b=\nabla U$.

\section{General Setting}\label{sec:sett}

As mentioned in the introduction, $\{Y_{n,k}\}_{n=0,1,\ldots;k=1,\ldots,2n+1}$ denotes a set of orthonormal spherical harmonics of degree $n$ and order $k$. Aside from the space  $L^2(\Omega_r)$ of square-integrable functions on $\Omega_r$, we also need the Sobolev space $\mathcal{H}_s(\Omega_r)$, $s\geq 0$. It is defined by
\begin{align}\label{eqn:sobspace}
\mathcal{H}_s(\Omega_r)=\left\{F\in L^2(\Omega_r):\|F\|_{\mathcal{H}_s(\Omega_r)}^2=\sum_{n=0}^\infty\sum_{k=1}^{2n+1}\Big(n+\textnormal{\footnotesize $\frac{1}{2}$}\Big)^{2s} \big|F_r^\wedge(n,k)\big|^2<\infty\right\},
\end{align}
where $F_r^\wedge(n,k)$ denotes the Fourier coefficient of degree $n$ and order $k$, i.e.,
\begin{align}
F_r^\wedge(n,k)=\int_{\Omega_r}F(y)\frac{1}{r}Y_{n,k}\left(\frac{y}{|y|}\right)d\omega(y).
\end{align}
A further notion that we need is 
\begin{align}\label{polnj1}
\textnormal{Pol}_{N}=\left\{K(x,y)=\sum_{n=0}^{N}\sum_{k=1}^{2n+1}K^\wedge(n) Y_{n,k}\left(\xi\right)Y_{n,k}\left(\eta\right):K^\wedge(n)\in\mathbb{R}\right\},
\end{align}
the space of all band-limited zonal kernels with maximal degree $N$ (as always, $\xi$, $\eta$ denote the unit vectors $\frac{x}{|x|}$ and $\frac{y}{|y|}$, respectively). The kernels $\Phi_N$ and $\tilde{\Psi}_N$ from \eqref{eqn:phin} and \eqref{eqn:psin} are members of such spaces. Zonal means that $K$ only depends on the scalar product $\xi\cdot\eta$, more precisely,
\begin{align}\label{eqn:Kzonal}
K(x,y)=\sum_{n=0}^{N}\frac{2n+1}{4\pi}K^\wedge(n) P_n\left(\xi\cdot\eta\right),
\end{align}
with $P_n$ being the Legendre polynomial of degree $n$ (the expressions \eqref{polnj1} and \eqref{eqn:Kzonal} are connected by the spherical addition theorem). Thus, instead of $K(\cdot,\cdot)$ acting on $\Omega_r\times\Omega_R$ or $\Omega_r\times\Omega_r$, it can also be regarded as a function $K(\cdot)$ acting on the interval $[-1,1]$. In both cases we just write $K$. 

\subsection{Downward Continuation}\label{subsec:dc}

We return to Equations \eqref{eqn:11}--\eqref{eqn:13} in order to derive the approximation \eqref{eqn:firstapprox}, reminding that we are interested in $U^+=U|_{\Omega_r}$ (or $U^+=U|_{\Gamma_r}$, respectively). We start by considering only the equations \eqref{eqn:11} and \eqref{eqn:12}, leading to the reconstruction of $U^+$ from knowledge of $F_1$ on $\Omega_R$, $r<R$. Opposed to this, the determination of $F_1$ from $U^+$ is known as upward continuation. The operator $T^{up}:L^2(\Omega_r)\to L^2(\Omega_R)$, given by
\begin{align}\label{eqn:TrR}
F_1(x)={T}^{up}[U^+](x)=\int_{\Omega_r}K^{up}(x,y) U^+(y)d\omega(y),\quad x\in\Omega_R,
\end{align}
with 
\begin{align}\label{eqn:KrR}
K^{up}(x,y)&=\sum_{n=0}^\infty\sum_{k=1}^{2n+1}\sigma_n\frac{1}{R}Y_{n,k}\left(\xi\right)\frac{1}{r}Y_{n,k}\left(\eta\right)
\end{align}
and $\sigma_n=\left(\frac{r}{R}\right)^n$, describes this process. The downward continuation operator $T^{down}$ and acts in the following way:  
\begin{align}\label{eqn:TRr}
U^+(x)={T}^{down}[F_1](x)=\int_{\Omega_R}K^{down}(x,y) U^+(y)d\omega(y),\quad x\in\Omega_r,
\end{align}
with 
\begin{align}\label{eqn:KRr}
K^{down}(x,y)&=\sum_{n=0}^\infty\sum_{k=1}^{2n+1}\frac{1}{\sigma_n}\frac{1}{r}Y_{n,k}\left(\xi\right)\frac{1}{R}Y_{n,k}\left(\eta\right).
\end{align}
One way to deal with the unboundedness of $T^{down}$ is a multiscale representation where $T^{down}$ is approximated by a sequence of bounded operators. Following the course of \cite{freeden99} and \cite{schneider98}, we assume $\Phi_N$ to be a scaling kernel of the form \eqref{eqn:phin} with truncation index $\lfloor\kappa N\rfloor$ and symbols $\Phi_N^\wedge(n)$ that satisfy
\begin{itemize}
\item[(a)] $\lim_{N\to\infty}\Phi_N^\wedge(n)=\frac{1}{\sigma_n}$, uniformly with respect to $n=0,1,\ldots$,
\item[(b)] $\sum_{n=0}^\infty \frac{2n+1}{4\pi} \big|\Phi_N^\wedge(n)\big|< \infty$,  for all $N=0,1,\ldots$.
\end{itemize}
The bounded scaling transform $T_N:L^2(\Omega_R)\to L^2(\Omega_r)$ is then defined via 
\begin{align}\label{eqn:TjRr}
{{T}_N[F_1](x)=\int_{\Omega_R}\Phi_N(x,y) F_1(y)d\omega(y),\quad x\in\Omega_r,}
\end{align}
and represents an approximation of $T^{down}[F_1]$. This operator can be refined further by use of the wavelet transform
\begin{align}\label{eqn:WjRr}
{W}_N[F_1](x)=\int_{\Omega_R}\Psi_N(x,y) F_1(y)d\omega(y),\quad x\in\Omega_r,
\end{align}
where the kernel $\Psi_N$ is of the form \eqref{eqn:psin} with symbols $\Psi_N^\wedge(n)=\Phi_{\lfloor\kappa N\rfloor}^\wedge(n)-\Phi_N^\wedge(n)$. An approximation of $T^{down}[F_1]$ at the higher scale $\lfloor\kappa N\rfloor$, for some fixed $\kappa>1$, is then given by 
\begin{align}\label{eqn:multirep1}
{T}_{\lfloor\kappa N\rfloor}[F_1](x)={T}_N[F_1](x)+W_N[F_1](x),\quad x\in\Omega_r.
\end{align}
It has to be noted that the kernel $\Psi_N$ and the wavelet transform $W_N$ lack a tilde (as opposed to representations \eqref{eqn:firstapprox} and \eqref{eqn:psin}). This indicates that we have not taken data $F_2$ in $\Gamma_r$ into account yet. Operators and kernels with a tilde mean that information is mapped from $\Gamma_r$ to $\Gamma_r$ while a lack of the tilde typically indicates the mapping of information from $\Omega_R$ to $\Omega_r$ (or $\Gamma_r$, respectively). 

\subsection{Combination of Downward Continuation and Local Data}\label{subsec:comb}

In order to incorporate data $F_2$ in $\Gamma_r$ by use of a wavelet transform, it is necessary to rewrite \eqref{eqn:WjRr}. Observing that $F_1$ and $F_2$ are only specific expressions of $U$ on the spheres $\Omega_R$ and $\Omega_r$, respectively, we find 
\begin{align}\label{eqn:Wjrr}
&\int_{\Omega_R}\Psi_N(x,y) F_1(y)d\omega(y)=\int_{\Omega_R}\Psi_N(x,y) U(y)d\omega(y)
\\&=\int_{\Omega_r}\tilde{\Psi}_N(x,y) U(y)d\omega(y)=\int_{\Omega_r}\tilde{\Psi}_N(x,y) F_2(y)d\omega(y),\quad x\in\Omega_r,\nonumber
\end{align}
where $\tilde{\Psi}_N$ is of the form \eqref{eqn:psin} with  $\tilde{\Psi}_N^\wedge(n)=\sigma_n{\Psi}_N^\wedge(n)=\Phi_{\lfloor\kappa N\rfloor}^\wedge(n)\sigma_n-\Phi_N^\wedge(n)\sigma_n$ (the factor $\sigma_n$ stems from the downward continuation that occurs in the second equality of \eqref{eqn:Wjrr}). We slightly modify the symbol $\tilde{\Psi}_N^\wedge(n)$ by use of the auxiliary symbol $\tilde{\Phi}_N^\wedge(n)$, so that it reads  
\begin{align}
{\tilde{\Psi}_N^\wedge(n)=\tilde{\Phi}_N^\wedge(n)-\Phi_N^\wedge(n)\sigma_n.}
\end{align}
This has the effect that now two parameters are available, namely $\Phi_N^\wedge(n)$, which reflects the behaviour of the operator $T_N$ responsible for the downward continuation, and $\tilde{\Phi}_N^\wedge(n)$, which offers a chance to control the localization of $\tilde{\Psi}_N$ and the behaviour of $\tilde{W}_N$ to a certain amount. The auxiliary symbol needs to satisfy
\begin{itemize}
\item[(a')] $\lim_{N\to\infty}\tilde{\Phi}_N^\wedge(n)=1$, uniformly with respect to $n=0,1,\ldots$,
\item[(b')] $\sum_{n=0}^\infty \frac{2n+1}{4\pi} \big|\tilde{\Phi}_N^\wedge(n)\big|< \infty$,  for all $N=0,1,\ldots$.
\end{itemize}
Remembering that $F_2$ is only available locally in $\Gamma_r$ and paying tribute to \eqref{eqn:Wjrr}, we define the wavelet transform 
\begin{align}\label{eqn:Wjloc}
{\tilde{W}_N[F_2](x)=\int_{\mathcal{C}_r(x,\rho)}\tilde{\Psi}_N(x,y) F_2(y)d\omega(y),\quad x\in\tilde{\Gamma}_r.}
\end{align}
$\mathcal{C}_r(x,\rho)$ denotes the spherical cap $\{y\in\Omega_r:1-\frac{x}{|x|}\cdot\frac{y}{|y|}<\rho\}$ with radius $\rho\in(0,2)$ and center $x\in\Omega_r$. The subset $\tilde{\Gamma}_r\subset\Gamma_r$ is chosen such that $\mathcal{C}_r(x,\rho)\subset\Gamma_r$ for every $x\in\tilde{\Gamma}_r$ and some $\rho\in(0,2)$ that is fixed in advance. The restriction to spherical caps is somewhat artificial and serves the sole purpose of simplifying the optimization of $\tilde{\Psi}_N$ (for the actual numerical evaluation of $U^+$ later on, we integrate over all of $\Gamma_r$ to make use of all available data). Summing up, the relations \eqref{eqn:TjRr}--\eqref{eqn:Wjloc} motivate
\begin{align}\label{eqn:firstapprox2}
{U_N={T}_N[F_1]+{\tilde{W}}_N[F_2]}
\end{align}
as an approximation of $U^+$ in $\tilde{\Gamma}_r$ (compare \eqref{eqn:firstapprox} in the introduction).

\section{The Minimizing Functional}\label{sec:min}

From now on we assume contaminated input data $F_1^{\eps_1}=F_1+\eps_1E_1$ and $F_2^{\varepsilon_2}=F_2+\varepsilon_2E_2$ with deterministic noise $E_1\in L^2(\Omega_R)$, $E_2\in L^2(\Omega_r)$, and $\eps_1,\eps_2>0$. The approximation \eqref{eqn:firstapprox2} of $U^+$ in $\tilde{\Gamma}_r$ is then modified by
\begin{align}\label{eqn:firstapprox3}
{U_N^\eps={T}_N[F_1^{\eps_1}]+{\tilde{W}}_N[F_2^{\eps_2}],}
\end{align}
where $\eps$ stands short for $(\eps_1,\eps_2)^T$. It is the aim of this paper to find kernels $\Phi_N$ and $\tilde{\Psi}_N$ (determined by the symbols $\Phi_N^\wedge(n)$ and $\tilde{\Psi}_N^\wedge(n)=\tilde{\Phi}_N^\wedge(n)-\Phi_N^\wedge(n)\sigma_n$, respectively) that keep the error $\|U^+-U_N^\eps\|_{L^2(\tilde{\Gamma}_r)}$ small and allow some adaptations to possible a-priori knowledge on $F_1^{\eps_1}$ and $F_2^{\eps_2}$ (such as noise level of the measurements or data density). We start with the estimate
\begin{align}
&\|U^+ - U_N^{\varepsilon}\|_{L^2(\tilde{\Gamma}_r)}\label{eqn:errorest2}
\\&\leq \|U^+ - {T}_{N}[F_1]-\tilde{W}_{N}[F_2]\|_{L^2(\tilde{\Gamma}_r)}+\|{T}_{N}[F_1-F_1^{\varepsilon_1}]+\tilde{W}_{N}[F_2-F_2^{\varepsilon_2}]\|_{L^2(\tilde{\Gamma}_r)}\nonumber
\\&\leq \|(1-{T}_{N}{T}^{up}-\tilde{W}_{N})[U^+]\|_{L^2(\tilde{\Gamma}_r)}+\varepsilon_1\|{T}_{N}[E_1]\|_{L^2(\tilde{\Gamma}_r)}+\varepsilon_2\|\tilde{W}_{N}[E_2]\|_{L^2(\tilde{\Gamma}_r)}.\nonumber
\end{align}
The first term on the right hand side can be split up further in the following way:
\begin{align}
& \|(1-{T}_{N}{T}^{up}-\tilde{W}_{N})[U^+]\|_{L^2(\tilde{\Gamma}_r)}\label{eqn:errorest22}
\\&\leq \frac{1}{2}\|(1-{T}_{N}{T}^{up})[U^+]\|_{L^2(\tilde{\Gamma}_r)}+ \frac{1}{2}\left\|\int_{\Omega_r}\tilde{\Psi}_N(\cdot,y)U^+(y)d\omega(y)\right\|_{L^2(\tilde{\Gamma}_r)}\nonumber
\\&\quad\, +\frac{1}{2}\left\|U^+(x)-\int_{\Omega_r}\tilde{\Phi}_N(\cdot,y)U^+(y)d\omega(y)\right\|_{L^2(\tilde{\Gamma}_r)}+ \left\|\int_{\Omega_r\setminus\mathcal{C}_r(\cdot,\rho)}\tilde{\Psi}_N(\cdot,y)U^+(y)d\omega(y)\right\|_{L^2(\tilde{\Gamma}_r)}.\nonumber
\end{align}
The last term on the right hand side of \eqref{eqn:errorest22} simply compensates the extension of the integration region in the two preceding terms from $\mathcal{C}_r(x,\rho)$ to all of $\Omega_r$. We continue with 
\begin{equation}\label{eqn:errorest23}
\begin{aligned}
& \|(1-{T}_{N}{T}^{up}-\tilde{W}_{N})[U^+]\|_{L^2(\tilde{\Gamma}_r)}
\\&\leq \frac{1}{2}\left\|\sum_{n=0}^\infty\sum_{k=1}^{2n+1}(1-{\Phi}_N^\wedge(n)\sigma_n)\big(U_r^+\big)^\wedge(n,k)\frac{1}{r}Y_{n,k}\right\|_{L^2(\Omega_r)}
\\&\quad+\frac{1}{2}\left\|\sum_{n=0}^\infty\sum_{k=1}^{2n+1}\tilde{\Psi}_{N}^\wedge(n)\big(U_r^+\big)^\wedge(n,k)\frac{1}{r}Y_{n,k}\right\|_{L^2(\Omega_r)}
\\&\quad+\frac{1}{2}\left\|\sum_{n=0}^\infty\sum_{k=1}^{2n+1}(1-\tilde{\Phi}_N^\wedge(n))\big(U_r^+\big)^\wedge(n,k)\frac{1}{r}Y_{n,k}\right\|_{L^2(\Omega_r)}
\\&\quad+\left\|\,\int_{\Omega_r\setminus\mathcal{C}_r(\rho,\cdot)}\tilde{\Psi}_{N}(\cdot,y) U^+(y)d\omega(y)\right\|_{L^2(\Omega_r)}
\\&\leq\sup_{n=0,1,\ldots}\frac{\big|1-\tilde{\Phi}_N^\wedge(n)\big|}{2\big(n+\frac{1}{2}\big)^s}\|U^+\|_{\mathcal{H}_s(\Omega_r)}+\sup_{n=0,1,\ldots}\frac{\big|1-{\Phi}_{N}^\wedge(n)\sigma_n\big|}{2\big(n+\frac{1}{2}\big)^s}\|U^+\|_{\mathcal{H}_s(\Omega_r)}
\\&\quad+\frac{1}{2}\sup_{n=0,1,\ldots}\big|\tilde{\Psi}_{N}^\wedge(n)\big|\|U^+\|_{L^2(\Omega_r)}+2\sqrt{2}\pi r^2\big\|\tilde{\Psi}_{N}\big\|_{L^2([-1,1-\rho])}\|U^+\|_{L^2(\Omega_r)}.
\end{aligned}
\end{equation}
For the last estimate on the right hand side, we observe that, due to the zonality of the kernels, $\sup_{x\in\Omega_r}\|\tilde{\Psi}_{N}(x,\cdot)\|_{L^2(\Omega_r\setminus\mathcal{C}_r(x,\rho))}$ coincides with $\sqrt{2\pi} r\|\tilde{\Psi}_{N}\|_{L^2([-1,1-\rho])}$, where
\begin{align}\label{eqn:1dint}
\big\|\tilde{\Psi}_{N}\big\|_{L^2([-1,1-\rho])}^2=\int_{-1}^{1-\rho}|\tilde{\Psi}_{N}(t)|^2dt.
\end{align}
Similar estimates can be obtained for the terms $\varepsilon_1\|{T}_{N}[E_1]\|_{L^2(\tilde{\Gamma}_r)}$ and $\varepsilon_2\|\tilde{W}_{N}[E_2]\|_{L^2(\tilde{\Gamma}_r)}$ in \eqref{eqn:errorest2}, so that we end up with an overall estimate
\begin{align}
&\|U^+ - U_N^{\varepsilon}\|_{L^2(\tilde{\Gamma}_r)}\label{eqn:errorestfinal}
\\&\leq\|U^+\|_{\mathcal{H}_s(\Omega_r)}\sup_{n=0,1,\ldots}\frac{\big|1-\tilde{\Phi}_N^\wedge(n)\big|}{2\big(n+\frac{1}{2}\big)^s}+\|U^+\|_{\mathcal{H}_s(\Omega_r)}\sup_{n=0,1,\ldots}\frac{\big|1-{\Phi}_{N}^\wedge(n)\sigma_n\big|}{2\big(n+\frac{1}{2}\big)^s}\nonumber
\\&\quad+\varepsilon_1\|E_1\|_{L^2(\Omega_R)}\sup_{n=0,1,\ldots}\big|\Phi_{N}^\wedge(n)\big|
+\Big(\eps_2\|E_2\|_{L^2(\Omega_r)}+\frac{1}{2}\|U^+\|_{L^2(\Omega_r)}\Big)\sup_{n=0,1,\ldots}\big|\tilde{\Psi}_{N}^\wedge(n)\big|\nonumber
\\&\quad+2\sqrt{2}\pi r^2\left(\eps_2\|E_2\|_{L^2(\Omega_r)}+\|U^+\|_{L^2(\Omega_r)}\right)\big\|\tilde{\Psi}_{N}\big\|_{L^2([-1,1-\rho])}.\nonumber
\end{align}
Eventually, finding 'good' kernels $\Phi_N$ and $\tilde{\Psi}_N$ reduces to finding symbols $\Phi_N^\wedge(n)$, $\tilde{\Phi}_N^\wedge(n)$ that keep the right hand side of \eqref{eqn:errorestfinal} small (note that $\tilde{\Psi}_{N}^\wedge(n)$ is given by $\tilde{\Phi}_N^\wedge(n)-\Phi_N^\wedge(n)\sigma_n$). We choose these symbols to be the minimizers of the functional 
\begin{equation}\label{eqn:mineq2}
{\begin{aligned}
\mathcal{F}(\Phi_N,\tilde{\Psi}_N)=&\sum_{n=0}^{\lfloor\kappa N\rfloor}\tilde{\alpha}_{N,n}\big|1-\tilde{\Phi}_{N}^\wedge(n)\big|^2+\sum_{n=0}^{N}\alpha_{N,n}\big|1-{\Phi}_{N}^\wedge(n)\sigma_n\big|^2
\\&+\beta_N\sum_{n=0}^{N}\big|\Phi_N^\wedge(n)\big|^2+8\pi^2 r^4\big\|\tilde{\Psi}_N\big\|^2_{L^2([-1,1-\rho])},
\end{aligned}}
\end{equation}
with $\Phi_N$ being a member of Pol$_{N}$ and $\tilde{\Psi}_N$ a member of Pol$_{\lfloor\kappa N\rfloor}$. The suprema in \eqref{eqn:errorestfinal} have been changed to square sums to simplify the determination of the minimizers. All pre-factors appearing in \eqref{eqn:errorestfinal} have been compensated into the parameters $\tilde{\alpha}_{N,n}$, ${\alpha}_{N,n}$, and $\beta_N$. They decide how much emphasis is set on the approximation property, how much on the behaviour of the downward continuation, and how much on the localization of the kernel $\tilde{\Psi}_N$. More precisely, the first term on the right hand side of \eqref{eqn:mineq2} reflects the overall approximation error (under the assumption that undisturbed global data is available on $\Omega_R$ as well as on $\Omega_r$), the second term only measures the error due to the downward continuation of undisturbed data on $\Omega_R$. The third and fourth term can be regarded as penalty terms reflecting the norm of the regularized downward 
continuation operator $T_N$ and the localization of the wavelet kernel (i.e., the error made by neglecting information outside the spherical cap $\mathcal{C}_r(x,\rho)$), respectively.

\section{Theoretical Results}\label{sec:theo}

Sections \ref{sec:sett} and \ref{sec:min} have motivated the choice of the functional $\mathcal{F}$ (cf. \eqref{eqn:mineq2}) and the approximation $U_N^\eps$ of $U^+$ (cf. \eqref{eqn:firstapprox3}). In this section, we want to study the approximation more rigorously with respect to its convergence. The general idea for the proof of the convergence stems from \cite{michel11} where the optimization of approximate identity kernels has been treated. We start with a lemma indicating the solution of the minimization of the functional $\mathcal{F}$.

\begin{lem}\label{prop:minsol}
Assume that all parameters $\tilde{\alpha}_{N,n}$, $\alpha_{N,n}$, and $\beta_N$ are positive. Then there exist unique minimizers $\Phi_N\in\textnormal{Pol}_{N}$ and $\tilde{\Psi}_N\in\textnormal{Pol}_{\lfloor\kappa N\rfloor}$ of the functional $\mathcal{F}$ in \eqref{eqn:mineq2} that are determined by the symbols $\phi=(\Phi_N^\wedge(1)\sigma_1,\ldots,\Phi_N^\wedge(N)\sigma_{N},\tilde{\Phi}_{N}^\wedge(1),\ldots,\tilde{\Phi}_{N}^\wedge(\lfloor\kappa N\rfloor))^T$ which solve the linear equations
\begin{align}\label{eqn:lineq1}
\mathbf{M}\phi=\alpha,
\end{align}
where 
\begin{align}
\mathbf{M}=\left(\begin{array}{c|c}\mathbf{D}_1+\mathbf{P}_1&-\mathbf{P}_2\\\hline-\mathbf{P}_3&\mathbf{D}_2+\mathbf{P}_4\end{array}\right),
\end{align}
and $\alpha=(\alpha_{N,1},\ldots,\alpha_{N,N},\tilde{\alpha}_{N,1},\ldots,\tilde{\alpha}_{N,\lfloor\kappa N\rfloor})^T$. The diagonal matrices $\mathbf{D}_1$, $\mathbf{D}_2$ are given by 
\begin{align}
\mathbf{D}_1=\textnormal{diag}\left(\frac{\beta_N}{\sigma_n^2}+\alpha_{N,n}\right)_{n=0,\ldots,N}, \qquad\mathbf{D}_2=\textnormal{diag}\big(\tilde{\alpha}_{N,n}\big)_{n=0,\ldots,\lfloor\kappa N\rfloor},
\end{align}
whereas $\mathbf{P}_1$,\ldots, $\mathbf{P}_4$ are submatrices of the Gram matrix $\big(P_{n,m}^\rho\big)_{n,m=0,\ldots, \lfloor\kappa N\rfloor}$. More precisely,
\begin{align}
\mathbf{P}_1=\big(P_{n,m}^\rho\big)_{n,m=0,\ldots,N},&\qquad \mathbf{P}_2=\big(P_{n,m}^\rho\big)_{n=0,\ldots,N\atop m=0,\ldots,\lfloor\kappa N\rfloor},
\\\mathbf{P}_3=\big(P_{n,m}^\rho\big)_{n=0,\ldots,\lfloor\kappa N\rfloor\atop m=0,\ldots,N},&\qquad\mathbf{P}_4=\big(P_{n,m}^\rho\big)_{n,m=0,\ldots,\lfloor\kappa N\rfloor}.
\end{align}
with
\begin{align}
P_{n,m}^\rho=\frac{(2n+1)(2m+1)}{2}\int_{-1}^{1-\rho}P_n(t)P_m(t)dt.
\end{align}
\end{lem}

\begin{proof}
First we observe that the representation \eqref{eqn:psin} of the kernel $\tilde{\Psi}_N$ and its zonality together with the addition theorem for spherical harmonics imply 
\small
\begin{align}\label{eqn:psinorm}
&8\pi^2 r^4\left\|\tilde{\Psi}_N\right\|^2_{L^2([-1,1-\rho])}
\\&=\frac{8\pi^2 r^4}{r^4}\int_{-1}^{1-\rho}\Bigg|\sum_{n=0}^{\lfloor\kappa N\rfloor}\frac{2n+1}{4\pi}\left(\tilde{\Phi}_{N}^\wedge(n)-\Phi_N^\wedge(n)\sigma_n\right)P_n(t)\Bigg|^2dt\nonumber
\\&=\sum_{n=0}^{\lfloor\kappa N\rfloor}\sum_{m=0}^{\lfloor\kappa N\rfloor}\frac{(2n+1)(2m+1)}{2}\left(\tilde{\Phi}_{N}^\wedge(m)-\Phi_N^\wedge(m)\sigma_m\right)\left(\tilde{\Phi}_{N}^\wedge(n)-\Phi_N^\wedge(n)\sigma_n\right)\int_{-1}^{1-\rho}\!\!\!\!\!\!\!P_n(t)P_m(t)dt\nonumber
\\&=\sum_{n=0}^{\lfloor\kappa N\rfloor}\sum_{m=0}^{\lfloor\kappa N\rfloor}\left(\tilde{\Phi}_{N}^\wedge(m)-\Phi_N^\wedge(m)\sigma_m\right)\left(\tilde{\Phi}_{N}^\wedge(n)-\Phi_N^\wedge(n)\sigma_n\right)P_{n,m}^\rho.\nonumber
\end{align}
\normalsize
Inserting \eqref{eqn:psinorm} into \eqref{eqn:mineq2} and then differentiating the whole expression with respect to $\Phi_N^\wedge(n)\sigma_n$ and $\tilde{\Phi}_{N}^\wedge(n)$ leads to 
\begin{align}
-2\alpha_{N,n}(1-\Phi_N^\wedge(n)\sigma_n)+2\beta_N\frac{\Phi_N^\wedge(n)\sigma_n}{\sigma_n^2}-2\sum_{m=0}^{\lfloor\kappa N\rfloor}\left(\tilde{\Phi}_{N}^\wedge(m)-\Phi_N^\wedge(m)\sigma_m\right)P_{n,m}^\rho,
\end{align} 
for $n=0,\ldots,N$, and
\begin{align}
-2\tilde{\alpha}_{N,n}(1-\tilde{\Phi}_{N}^\wedge(n))+2\sum_{m=0}^{\lfloor\kappa N\rfloor}\left(\tilde{\Phi}_{N}^\wedge(m)-\Phi_N^\wedge(m)\sigma_m\right)P_{n,m}^\rho,
\end{align}
for $n=0,\ldots,\lfloor\kappa N\rfloor$, respectively. Setting the two expressions above equal to zero, a proper reordering leads to the linear equation \eqref{eqn:lineq1}. At last, we observe that the matrix $(P_{n,m}^\rho)_{n,m=0,\ldots,\lfloor\kappa N\rfloor}$ is positive definite (since it represents a Gram matrix of linearly independent functions) and that all appearing diagonal matrices are positive definite due to positive matrix entries. Thus, the matrix $\mathbf{M}$ is positive definite and the linear system \eqref{eqn:lineq1} is uniquely solvable and leads to a minimum of \eqref{eqn:mineq2}.
\end{proof}

We continue with the statement of convergence for $U_N^\varepsilon$ (cf. Theorem \ref{thm:convres}). For the proof we need a localization result for Shannon-type kernels, more precisely, a (spherical) variation of the Riemann localization property. We borrow this result as a particular case from \cite{yuguang13}:

\begin{prop}\label{prop:loc1} 
If $F\in\mathcal{H}_s(\Omega_r)$, $s\geq 2$, it holds that
\begin{align}\label{eqn:locprop}
\lim_{N\to\infty}\left\|\int_{\Omega_r\setminus\mathcal{C}_r(\rho,\cdot)}\tilde{\Psi}^{Sh}_N(\cdot,y)F(y)d\omega(y)\right\|_{L^2 (\Omega_r)}=0,
\end{align}
where $\tilde{\Psi}^{Sh}_N$ is the Shannon-type kernel with symbols $\big(\tilde{\Psi}_N^{Sh}\big)^\wedge(n)=1$, if $N +1\leq n\leq \lfloor\kappa N\rfloor$, and $\big(\tilde{\Psi}_N^{Sh}\big)^\wedge(n)=0$, else.
\end{prop}

\begin{thm}\label{thm:convres}
Assume that the parameters $\alpha_{N,n}$, $\tilde{\alpha}_{N,n}$, and $\beta_N$ are positive and suppose that, for some fixed $\delta>0$ and $\kappa>1$,
\begin{align}
\inf_{n=0,\ldots, \lfloor\kappa N\rfloor}\frac{\beta_N\frac{1-\left(\frac{R}{r}\right)^{2(N+1)}}{1-\left(\frac{R}{r}\right)^2}+( \lfloor\kappa N\rfloor+1)^2}{{\alpha}_{N,n}}&=\mathcal{O}\left(N^{-2(1+\delta)}\right),\quad \textit{ for }N\to\infty.\label{eqn:ass1}
\end{align}
The same relation shall hold true for $\tilde{\alpha}_{N,n}$. Additionally, let every $N$ be associated with an $\varepsilon_1=\varepsilon_1(N)>0$ and an $\varepsilon_2=\varepsilon_2(N)>0$ such that 
\begin{align}
\lim_{N\to\infty}\varepsilon_1 \left(\frac{R}{r}\right)^{N}=\lim_{N\to\infty}\varepsilon_2 \,{N}=0.\label{eqn:propx}
\end{align}
The functions $F_1:\Omega_R\to\mathbb{R}$ and $F_2:\Gamma_r\to\mathbb{R}$, $r<R$, are supposed to be such that a unique solution $U$ of \eqref{eqn:11}--\eqref{eqn:13} exists and that the restriction $U^+$ is of class $\mathcal{H}_s(\Omega_r)$, for some fixed $s\geq 2$. The erroneous input data is given by $F_1^{\eps_1}=F_1+\eps_1 E_1$ and $F_2^{\eps_2}=F_2+\eps_2 E_2$, with $E_1\in L^2(\Omega_R)$ and $E_2\in L^2(\Omega_r)$. If the kernels $\Phi_N\in\textnormal{Pol}_{N}$ and kernel $\tilde{\Psi}_N\in\textnormal{Pol}_{ \lfloor\kappa N\rfloor}$ are the minimizers of the functional $\mathcal{F}$ from \eqref{eqn:mineq2} and if $U_N^\eps$ is given as in \eqref{eqn:firstapprox3}, then
\begin{align}
\lim_{N\to\infty}\left\|U^+-U_ {N}^{\varepsilon}\right\|_{L^2(\tilde{\Gamma}_r)}=0.\label{eqn:convres}
\end{align}
\end{thm}

\begin{proof}
As an auxiliary set of kernels, we define the Shannon-type kernels $\Phi_N^{Sh}$ and $\tilde{\Psi}_N^{Sh}$ via the symbols
\begin{align}
\big(\Phi_N^{Sh}\big)^\wedge(n)=\left\{\begin{array}{ll}\frac{1}{\sigma_n},&n\leq N,\\0,&\textnormal{else},\end{array}\right.\qquad \big(\tilde{\Phi}_{N}^{Sh}\big)^\wedge(n)=\left\{\begin{array}{ll}1,&n\leq \lfloor\kappa N\rfloor,\\0,&\textnormal{else},\end{array}\right.
\end{align}
and $\big(\tilde{\Psi}_N^{Sh}\big)^\wedge(n)=\big(\tilde{\Phi}_N^{Sh}\big)^\wedge(n)-\big(\Phi_N^{Sh}\big)^\wedge(n)\sigma_n$. $\Phi_N^{Sh}$ represents the kernel of the so-called truncated singular value decomposition for the downward continuation operator ${T}^{down}$. By using the addition theorem for spherical harmonics and properties of the Legendre polynomials, we obtain
\begin{align}
\mathcal{F}(\Phi_N^{Sh},\tilde{\Psi}_N^{Sh})&=\beta_N\sum_{n=0}^{N}\frac{1}{\sigma_n^2}+8\pi^2 r^4\big\|\tilde{\Psi}_N^{Sh}\big\|^2_{L^2([-1,1-\rho])}\label{eqn:shannonest}
\\&\leq \beta_N\frac{1-\left(\frac{R}{r}\right)^{2(N+1)}}{1-\left(\frac{R}{r}\right)^{2}}+8\pi^2\sum_{n=0}^{\lfloor\kappa N\rfloor}\left(\frac{2n+1}{4\pi}\right)^2\left\|P_n\right\|^2_{L^2([-1,1])}\nonumber
\\&= \beta_N\frac{1-\left(\frac{R}{r}\right)^{2(N+1)}}{1-\left(\frac{R}{r}\right)^{2}}+(\lfloor\kappa N\rfloor+1)^2.\nonumber
\end{align}
The kernels that minimize $\mathcal{F}$ are denoted by $\Phi_N$ and $\tilde{\Psi}_N$. In consequence,
\begin{align}
\tilde{\alpha}_{N,n}\big(1-\tilde{\Phi}_{N}^\wedge(n)\big)^2&\leq\mathcal{F}(\Phi_N,\tilde{\Psi}_N)\leq \mathcal{F}(\Phi_N^{Sh},\tilde{\Psi}_N^{Sh})
\\&\leq \beta_N\frac{1-\left(\frac{R}{r}\right)^{2(N+1)}}{1-\left(\frac{R}{r}\right)^{2}}+(\lfloor\kappa N\rfloor+1)^2,\nonumber
\end{align}
for all $n\leq \lfloor\kappa N\rfloor$. In combination with \eqref{eqn:ass1}, this leads to
\begin{align}\label{eqn:conphij}
\big|1-\tilde{\Phi}_{N}^\wedge(n)\big|= \mathcal{O}(N^{-1-\delta}),
\end{align}
uniformly for all $n\leq \lfloor\kappa N\rfloor$ and $N\to\infty$. The estimate $|1-\Phi_N^\wedge(n)\sigma_n|=\mathcal{O}(N^{-1-\delta})$ follows in the same manner and holds uniformly for all $n\leq N$ and $N\to\infty$. Finally, since $\tilde{\Phi}_{N}^\wedge(n)=0$ for $n> \lfloor\kappa N\rfloor$, we end up with
\begin{align}
\lim_{N\to\infty}\sup_{n=0,1,\ldots}\frac{|1-\tilde{\Phi}_{N}^\wedge(n)|}{(n+\frac{1}{2})^{s}}=0,
\end{align}
which shows that the first term on the right hand side of the error estimate \eqref{eqn:errorestfinal} vanishes. The second term can be treated analogously. 

Due to the uniform boundedness of $|\Phi_J^\wedge(n)\sigma_n|$, there exists some constant $C>0$ such that \eqref{eqn:propx} implies
\begin{align}
\lim_{N\to\infty}\eps_1\sup_{n=0,1,\ldots}\left|\Phi_N^\wedge(n)\right|\leq C\lim_{N\to\infty}\eps_1\left(\frac{R}{r}\right)^{N}=0.
\end{align}
so that the third term on the right hand side of \eqref{eqn:errorestfinal} vanishes as well. The fourth term does not vanish. However, taking a closer look at the derivation of this term, we see that it suffices to show that $\big\|\sum_{n=0}^\infty\sum_{k=1}^{2n+1}\tilde{\Psi}_{N}^\wedge(n)F_r^\wedge(n,k)\frac{1}{r}Y_{n,k}\big\|_{L^2(\Omega_r)}$ tends to zero (where $F$ stands for $U^+$ or $E_2$). Using the previous results and $|\tilde{\Psi}_N^\wedge(n)|\leq |1-{\Phi}_N^\wedge(n)\sigma_n|+|1-\tilde{\Phi}_N^\wedge(n)|$, it follows that
\begin{align}
\left\|\sum_{n=0}^\infty\sum_{k=1}^{2n+1}\tilde{\Psi}_{N}^\wedge(n)F_r^\wedge(n,k)\frac{1}{r}Y_{n,k}\right\|_{L^2(\Omega_r)}^2\leq \mathcal{O}(N^{-2\delta})+C\sum_{n=N+1}^{\lfloor \kappa N\rfloor}\sum_{k=1}^{2n+1}\left|F_r^\wedge(n,k)\right|^2,
\end{align}
for some constant $C>0$. Latter vanishes for $N\to\infty$ since $F$ (i.e., $U^+$ or $E_2$) is of class $L^2(\Omega_r)$. In order to handle the last term of the error estimate \eqref{eqn:errorestfinal}, it has to be shown that $\|\tilde{\Psi}_N\|_{L^2([-1,1-\rho])}$ tends to zero. This, again, is generally not true. But, taking a closer look at the derivation of the estimate indicates that it suffices to show that
\small\begin{align}\label{eqn:altterms}
\Bigg\|\,\,\int_{\Omega_r\setminus\mathcal{C}_r(\rho,\cdot)}\tilde{\Psi}_{N}(\cdot,y) U^+(y)d\omega(y)\Bigg\|_{L^2(\Omega_r)}, \qquad \eps_2\Bigg\|\,\,\int_{\Omega_r\setminus\mathcal{C}_r(\rho,\cdot)}\tilde{\Psi}_{N}(\cdot,y) E_2(y)d\omega(y)\Bigg\|_{L^2(\Omega_r)}
\end{align}\normalsize
vanish as $N\to\infty$. For the left expression we obtain
\begin{equation}\label{eqn:locest}
\begin{aligned}
&\Bigg\|\,\,\int_{\Omega_r\setminus\mathcal{C}_r(\rho,\cdot)}\tilde{\Psi}_{N}(\cdot,y) U^+(y)d\omega(y)\Bigg\|_{L^2(\Omega_r)}
\\&\leq\Bigg\|\,\,\int_{\Omega_r\setminus\mathcal{C}_r(\rho,\cdot)}\tilde{\Psi}^{Sh}_{N}(\cdot,y) U^+(y)d\omega(y)\Bigg\|_{L^2(\Omega_r)}
\\&\quad\,+\Bigg\|\,\,\int_{\Omega_r\setminus\mathcal{C}_r(\rho,\cdot)}\left(\tilde{\Psi}^{Sh}_{N}(\cdot,y)-\tilde{\Psi}_{N}(\cdot,y)\right) U^+(y)d\omega(y)\Bigg\|_{L^2(\Omega_r)}
\\&\leq\Bigg\|\,\,\int_{\Omega_r\setminus\mathcal{C}_r(\rho,\cdot)}\tilde{\Psi}^{Sh}_{N}(\cdot,y) U^+(y)d\omega(y)\Bigg\|_{L^2(\Omega_r)}
\\&\quad\,+2\pi r^2\|\tilde{\Psi}^{Sh}_{N}-\tilde{\Psi}_{N}\|_{L^2([-1,1-\rho])} \|U^+\|_{L^1(\Omega_r)},
\end{aligned}
\end{equation}
where Young's inequality has been used in the last row. Since $U^+\in \mathcal{H}_s(\Omega_r)$, $s\geq 2$, Proposition \ref{prop:loc1} implies that the first term on the right hand side of \eqref{eqn:locest} tends to zero as $N\to\infty$. The second term on the right hand side of \eqref{eqn:locest} can be treated as follows 
\begin{equation}
\begin{aligned}
&\big\|\tilde{\Psi}^{Sh}_{N}-\tilde{\Psi}_{N}\big\|_{L^2([-1,1-\rho])}^2
\\&\leq\sum_{n=0}^{\lfloor\kappa N\rfloor}\left(\frac{2n+1}{4\pi}\right)^2\left|\tilde{\Psi}_N^\wedge(n)-\left(\tilde{\Psi}^{Sh}_N\right)^\wedge(n)\right|^2\|P_n\|_{L^2([-1,1])}^2
\\&\leq\sum_{n=0}^{\lfloor\kappa N\rfloor}\frac{2n+1}{8\pi^2}\left|\left|\tilde{\Phi}_{N}^\wedge(n)-1\right|+\left|{\Phi}_N^\wedge(n)\sigma_n-1\right|\chi_{[0,N]}(n)\right|^2
\\&=\sum_{n=0}^{\lfloor\kappa N\rfloor}\frac{2n+1}{8\pi^2}\mathcal{O}\left(N^{-2(1+\delta)}\right)=\mathcal{O}\left(N^{-2\delta}\right).
\end{aligned}
\end{equation}
By $\chi_{[0,N]}(n)$ we mean the characteristic function, i.e., it is equal to 1 if $n=0,\ldots,N$, and equal to 0 otherwise. In conclusion, the left expression in \eqref{eqn:altterms} vanishes as $N\to\infty$. For the right expression in \eqref{eqn:altterms}, we obtain the desired result in a similar but easier manner. Finally, combining all steps of the proof, we have shown that the right hand side of the error estimate \eqref{eqn:errorestfinal} converges to zero, which yields \eqref{eqn:convres}. 
\end{proof}

\begin{rem}\label{rem:lockernels}
The condition $U^+\in \mathcal{H}_s(\Omega_r)$, $s\geq 2$, in Theorem \ref{thm:convres} can be relaxed to $U^+\in \mathcal{H}_s(\Omega_r)$, $s>0$, if the minimizing functional $\mathcal{F}$ in \eqref{eqn:mineq2} is substituted by
\begin{align}\label{eqn:mineq3}
\mathcal{F}_{Filter}(\Phi_N,\tilde{\Psi}_N)=&\sum_{n=0}^{\lfloor \kappa N\rfloor}\tilde{\alpha}_{N,n}\big|K_N^\wedge(n)-\tilde{\Phi}_{N}^\wedge(n)\big|^2+\sum_{n=0}^{N}\alpha_{N,n}\big|K_N^\wedge(n)-{\Phi}_{N}^\wedge(n)\sigma_n\big|^2
\\&+\beta_N\sum_{n=0}^{N}\big|\Phi_N^\wedge(n)\big|^2+8\pi^2 r^4\left\|\tilde{\Psi}_N\right\|^2_{L^2([-1,1-\rho])},\nonumber
\end{align}
where $K_N^\wedge(n)$ are the symbols of a filtered kernel (compare, e.g., \cite{yuguang132} and references therein for more information)
\begin{align}
\tilde{\Phi}_{N}^{Filter}(x,y)=\sum_{n=0}^{N}\sum_{k=1}^{2n+1}K_N^\wedge(n)Y_{n,k}\left(\frac{x}{|x|}\right)Y_{n,k}\left(\frac{y}{|y|}\right).
\end{align}
Then the proof of Theorem \ref{thm:convres} follows in a very similar manner as before, just that the minimizing kernels are not compared to the Shannon-type kernel but to the filtered kernel. Opposed to the Shannon-type kernel, an appropriately filtered kernel has the localization property $\lim_{N\to\infty}\|\tilde{\Phi}_{N}^{Filter}\|_{L^1([-1,1-\rho])}=0$  which makes the condition $s\geq 2$ on the smoothness of $U^+$ (required for Proposition \ref{prop:loc1}) obsolete. 
\end{rem}

To finish this section, we want to comment on the localization of the kernel $\tilde{\Psi}_N$. While we have used $\|\tilde{\Psi}_N\|_{L^2([-1,1-\rho])}$ as a measure for the localization inside a spherical cap $\mathcal{C}_r(\cdot,\rho)$ (values close to zero meaning a good localization, i.e., small leakage of information into $\Omega_r\setminus\mathcal{C}_r(\cdot,\rho)$), a more suitable quantity to consider would be
\begin{align}\label{eqn:altloc}
\frac{\big\|\tilde{\Psi}_N\big\|_{L^2([-1,1-\rho])}}{\big\|\tilde{\Psi}_N\big\|_{L^2([-1,1])}}.
\end{align}
This is essentially the expression that is minimized for the construction of Slepian functions (see, e.g., \cite{plattner13}, \cite{plattner13b}, \cite{simons06}, \cite{simons10}). However, using \eqref{eqn:altloc} as a penalty term in the functional $\mathcal{F}$ from \eqref{eqn:mineq2} would make it significantly harder to find its minimizers. Furthermore, it turns out that the kernel $\tilde{\Psi}_N$ that minimizes the original functional actually keeps the quantity \eqref{eqn:altloc} small as well (at least asymptotically in the sense that \eqref{eqn:altloc} vanishes for $N\to\infty$). The latter essentially originates from the property that $\tilde{\Psi}_N$ converges to a Shannon-type kernel (which has the desired property).

\begin{lem}\label{lem:loc}
Assume that the parameters $\alpha_{N,n}$, $\tilde{\alpha}_{N,n}$, and $\beta_N$ are positive and suppose that, for some fixed $\delta>0$ and $\kappa>1$,
\begin{align}
\inf_{n=0,\ldots, \lfloor\kappa N\rfloor}\frac{\beta_N\frac{1-\left(\frac{R}{r}\right)^{2(N+1)}}{1-\left(\frac{R}{r}\right)^2}+( \lfloor\kappa N\rfloor+1)^2}{{\alpha}_{N,n}}&=\mathcal{O}\left(N^{-2(1+\delta)}\right),\quad \textit{ for }N\to\infty.\label{eqn:ass11}
\end{align}
The same relation shall hold true for $\tilde{\alpha}_{N,n}$. If the scaling kernel $\Phi_N\in\textnormal{Pol}_{N}$ and the wavelet kernel $\tilde{\Psi}_N\in\textnormal{Pol}_{\lfloor\kappa N\rfloor}$ are the minimizers of the functional $\mathcal{F}$ from \eqref{eqn:mineq2}, then
\begin{align} \label{eqn:optloc2}
\lim_{N\to\infty}
\frac{\big\|\tilde{\Psi}_N\big\|^2_{L^2([-1,1-\rho])}}{\big\|\tilde{\Psi}_N\big\|^2_{L^2([-1,1])}}=0.
\end{align}
\end{lem}

\begin{proof}
The function $F_N(t)={|\tilde{\Psi}_N(t)|^2}/{\|\tilde{\Psi}_N\|^2_{L^2([-1,1])}}$ can be regarded as the density function of a random variable $t\in[-1,1]$. Thus, we can write
\begin{align}
\frac{\big\|\tilde{\Psi}_N\big\|^2_{L^2([-1,1-\rho])}}{\big\|\tilde{\Psi}_N\big\|^2_{L^2([-1,1])}}&=\int_{-1}^{1-\rho}F_N(t)dt=P_N(t< 1-\rho)=1-P_N\big(t-E_N(t)\geq 1-\rho-E_N(t)\big)\nonumber
\\&\leq \frac{V_N(t)}{V_N(t)+ (1-\rho-E_N(t))^2},
\end{align}
where we have used Cantelli's inequality for the last estimate. By $P_N(t< a)$ we mean the probability (with respect to the density function $F_N$) that $t$ lies in the interval $[-1,a)$ while $P_N(t\geq a)$ means the probability of $t$ being in the interval $[a,1]$. Furthermore, $E_N(t)$ denotes the expected value of $t$ and $V_N(t)=E(t^2)-(E(t))^2$ the variance of $t$. In other words, we are done if we can show that $\lim_{N\to\infty}E_N(t)=1$ and $\lim_{N\to \infty}V_N(t)=0$. 

First, we use the addition theorem for spherical harmonics and the recurrence relation  $tP_n(t)=\frac{1}{2n+1}\left((n+1)P_{n+1}(t)+nP_{n-1}(t)\right)$ to obtain  
\begin{align}\label{eqn:et1}
\int_{-1}^1 t|\tilde{\Psi}_N(t)|^2 dt&= \frac{1}{r^2} \sum_{n=0}^{\lfloor\kappa N\rfloor}\sum_{m=0}^{\lfloor\kappa N\rfloor}\frac{2n+1}{4\pi}\frac{2m+1}{4\pi}\tilde{\Psi}_N^\wedge(n)\tilde{\Psi}_N^\wedge(m)\int_{-1}^1tP_n(t)P_m(t)dt
\\&=\sum_{n=0}^{\lfloor\kappa N\rfloor}\frac{2n+1}{16\pi^2r^2}\tilde{\Psi}_N^\wedge(n)\left(n\tilde{\Psi}_N^\wedge(n-1)+(n+1)\tilde{\Psi}_N^\wedge(n+1)\right)\int_{-1}^1|P_n(t)|^2 dt\nonumber
\\&=\frac{1}{8\pi^2r^2}\sum_{n=0}^{\lfloor\kappa N\rfloor}n\tilde{\Psi}_N^\wedge(n)\tilde{\Psi}_N^\wedge(n-1)+(n+1)\tilde{\Psi}_N^\wedge(n)\tilde{\Psi}_N^\wedge(n+1).\nonumber
\end{align}
From \eqref{eqn:conphij} and the corresponding estimate for $\Phi_N^\wedge(n)\sigma_n$ we find $|\tilde{\Psi}_N^\wedge(n)|^2=\mathcal{O}(N^{-2(1+\delta)})$ for $n\leq N$, and $|1-\tilde{\Psi}_N^\wedge(n)|^2=|1-\tilde{\Phi}_{N}^\wedge(n)|^2=\mathcal{O}(N^{-2(1+\delta)})$, for $N+1\leq n\leq \lfloor\kappa N\rfloor$. As a consequence,  \eqref{eqn:et1} implies
\begin{align}\label{eqn:et11}
&\int_{-1}^1 t|\tilde{\Psi}_N(t)|^2 dt
\\&=\mathcal{O}(N^{-2\delta})+\sum_{n=N+1}^{\lfloor\kappa N\rfloor} \frac{2n+1}{8\pi^2r^2}|\tilde{\Psi}_N^\wedge(n)|^2\left(\frac{n}{2n+1}\frac{\tilde{\Psi}_N^\wedge(n-1)}{\tilde{\Psi}_N^\wedge(n)}+\frac{n+1}{2n+1}\frac{\tilde{\Psi}_N^\wedge(n+1)}{\tilde{\Psi}_N^\wedge(n)}\right),\nonumber
\end{align}
where $\sup_{N+1\leq n\leq\lfloor \kappa N\rfloor}\big|1-\big(\frac{n}{2n+1}{\tilde{\Psi}_N^\wedge(n-1)}/{\tilde{\Psi}_N^\wedge(n)}+\frac{n+1}{2n+1}{\tilde{\Psi}_N^\wedge(n+1)}/{\tilde{\Psi}_N^\wedge(n)}\big)\big|\leq\varepsilon_N$, for  some $\varepsilon_N>0$ that satisfies $\varepsilon_N\to0$ as $N\to\infty$. The term $\|\tilde{\Psi}_N\|_{L^2([-1,1])}^2$ can be expressed similarly by
\begin{align}\label{eqn:et2}
\|\tilde{\Psi}_N\|_{L^2([-1,1])}^2=\sum_{n=0}^{\lfloor\kappa N\rfloor} \frac{2n+1}{8\pi^2r^2}|\tilde{\Psi}_N^\wedge(n)|^2=\mathcal{O}(N^{-2\delta})+\sum_{n=N+1}^{\lfloor\kappa N\rfloor} \frac{2n+1}{8\pi^2r^2}|\tilde{\Psi}_N^\wedge(n)|^2.
\end{align}
Thus, combining \eqref{eqn:et1}--\eqref{eqn:et2}, we obtain
\begin{align}
\lim_{N\to\infty}E_N(t)=\lim_{N\to\infty}\int_{-1}^1 t|F_N(t)| dt=\lim_{N\to\infty}\frac{\int_{-1}^1 t|\tilde{\Psi}_N(t)|^2 dt}{\|\tilde{\Psi}_N\|_{L^2([-1,1])}^2}=1.
\end{align}
In a similar fashion it can be shown that $\lim_{N\to\infty}E_N(t^2)=1$, implying $\lim_{N\to\infty}V_N(t)=0$, which concludes the proof.
\end{proof}

\section{Numerical Test}\label{sec:num}

We generate a potential $U^+$ and data $F_1$, $F_2$ from uniformly distributed random Fourier coefficients up to spherical harmonic degree $100$. Analogously, the noise $E_1$, $E_2$ is produced by uniformly distributed random Fourier coefficients up to spherical harmonic degree $110$ and is then scaled such that $\|E_1\|_{L^2(\Omega_R)}=\|F_1\|_{L^2(\Omega_R)}$ and $\|E_2\|_{L^2(\Gamma_r)}=\|F_2\|_{L^2(\Gamma_r)}$. The mean Earth radius is given by $r=6371.2$km; for the satellite orbit we choose $R=7071.2$km (i.e., a satellite altitude of $700$km above the Earth's surface). We fix $N=100$ for the approximation $U_N^\varepsilon$ and choose $\kappa>0$ such that $\lfloor \kappa N\rfloor=130$. Then $U_N^\varepsilon$ is computed for the following different settings: 
\begin{itemize}
 \item[(a)] data $F_1^{\varepsilon_1}$ is available on an equiangular grid of $40804$ points on $\Omega_R$; data $F_2^{\varepsilon_2}$ is given on a Gauss-Legendre grid of $49062$ points in a spherical cap $\Gamma_r=\mathcal{C}_r(x_0,2\rho)$ around the North pole $x_0$. We choose $\tilde{\Gamma}_r=\mathcal{C}_r(x_0,\rho)$ and vary the radius $\rho\in\{0.5,0.02\}$, 
 \item[(b)] the parameters of the functional $\mathcal{F}$ are chosen among the cases $\beta_N\in\{10^{-3},\ldots,10^{4}\}$, $\tilde{\alpha}_{N,n}\in\{10^{1},\ldots, 10^{8}\}$, as well as ${\alpha}_{N,n}\in\{\tilde{\alpha}_{N,n},0.2\tilde{\alpha}_{N,n}\}$,
 \item[(c)] the noise level on $\Omega_R$ is varied among $\varepsilon_1\in\{0.001,0.01,0.05,0.1\}$; for the noise in $\Gamma_r$ we choose $\varepsilon_2\in\{\varepsilon_1,2\varepsilon_1,5\varepsilon_1\}$.
\end{itemize}
For the numerical integration required for the evaluation of $U_N^\varepsilon$ we use the scheme from \cite{drihea} on $\Omega_R$ and the scheme from \cite{womers12} in $\mathcal{C}_r(x,\rho)$. As a reference for the optimized kernels, an approximation $U_M^{\varepsilon,Sh}$ with Shannon-type kernels is computed as well: 
\begin{itemize}
 \item[(a)] we choose $\Phi_{M}^{Sh}$ with symbols $\big(\Phi_{M}^{Sh}\big)^\wedge(n)=\frac{1}{\sigma_n}$, if $n\leq M$, and $\big(\Phi_{M}^{Sh}\big)^\wedge(n)=0$, else, as well as $\tilde{\Psi}_{M}^{Sh}$ with symbols $\big(\tilde{\Psi}_{M}^{Sh}\big)^\wedge(n)=1$, if $M+1\leq n\leq 130$, and $\big(\tilde{\Psi}_{M}^{Sh}\big)^\wedge(n)=0$, else,
 \item[(b)] the cut-off degree $M$ is varied among $M\in\{0,2,4,\ldots,100\}$.
\end{itemize}
The Shannon-type kernels form a reasonable reference since the optimization of $\Phi_N$, $\tilde{\Psi}_N$ via $\mathcal{F}$ is done with respect to a Shannon-type situation. The relative errors $\|U_N^\varepsilon-U^+\|_{L^2(\tilde{\Gamma}_r)}/\|U^+\|_{L^2(\tilde{\Gamma}_r)}$ and $\|U_M^{\varepsilon,Sh}-U^+\|_{L^2(\tilde{\Gamma}_r)}/\|U^+\|_{L^2(\tilde{\Gamma}_r)}$ for the different situations are supplied in Tables \ref{tab:err1} and \ref{tab:err2}. 

As a comparison, we also compute approximations of $U^+$ using solely data $F_2$ in $\Gamma_r$ or solely data $F_1$ on $\Omega_R$, respectively. More precisely:
\begin{itemize}
 \item[(a)] spherical splines (based on the iterated Green function for the Beltrami operator; see, e.g., \cite{freeden81} and \cite{freeschrei09}) are applied for the approximation from local/regional data in $\Gamma_r$; for spherical cap radius $\rho=0.5$, we interpolate; for spherical cap radius $\rho=0.02$, we use a least square approach with smoothing parameter $\alpha=10^{-5}$ in order to deal with the ill-conditioned spline matrix (instead of a Gauss-Legendre grid, we choose approximately 30000 points in $\Gamma_r$ with similar spacing in longitudinal and latitudinal direction, avoiding concentration towards the poles and leading to better results for the spline method),
 \item[(b)] when using satellite data on $\Omega_R$ only, we apply the truncated singular value decomposition (TSVD) for different cut-off degrees $M\in\{0,2,4,\ldots,100\}$; in addition, the regularized collocation method (RCM) from \cite{naumov} with Tikhonov regularization is tested for different regularization parameters $\alpha\in\{10^{-7},\ldots,10^{-1}\}$ (the equiangular integration grid on $\Omega_R$ and the corresponding weights are the same as chosen before).
\end{itemize}

Tables \ref{tab:err4a} and \ref{tab:err4} indicate the results for the set of parameters that performed best for each of the methods above. We can make the following observations:
\begin{itemize}
\item[(1)] Tables  \ref{tab:err1} and \ref{tab:err2} show that the optimized kernels behave better than the Shannon-type kernels over the broad range of settings.
\item[(2)] The results in Table \ref{tab:err1} indicate the expected behaviour for the Shannon-type kernels: in the case $\varepsilon_1=\varepsilon_2$ only ground data is used (i.e., cut-off degree $M=0$), while the cut-off degree increases if the noise level of the ground data in $\Gamma_r$ is larger than for the satellite data on $\Omega_R$ (i.e., $\varepsilon_2>\varepsilon_1$). The same behaviour can be observed for the optimized kernels (cf. Figure \ref{fig:kernelplot}(a)), however, with a smoother transition between ground and satellite data: e.g., for the choice $\beta_n=3\cdot 10^3$, $\tilde{\alpha}_{N,n}=10^3$, ${\alpha}_{N,n}=10^3$, mostly ground data is used (i.e., $\tilde{\Psi}_N^\wedge(n)\approx 1$), but to a small extent (up to around spherical harmonic degree $n=10$) also satellite data is included (i.e., ${\Phi}_N^\wedge(n)\not=0$). This leads to an improvement of the relative error, e.g., by more than a factor two for the case $\eps_1=\eps_2=0.001$ (cf. Table \ref{tab:err1}(a)).
\item[(3)] If we reduce the spherical cap radius from $\rho=0.5$ to $\rho=0.02$, we find a similar behaviour as described in (2) (cf. Table \ref{tab:err2} and Figure \ref{fig:kernelplot}(b)). But we also find that for a small noise level $\varepsilon_1=0.001$ the relative error of the approximation is amplified significantly (by a factor of approximately $10$) while it behaves better for larger noise levels. Thus, the approximation error for small noise levels can be in first place accounted to an insufficient localization of the used kernels. Increasing the truncation degree $N$ of the scaling and wavelet kernels $\Phi_N$ and $\tilde{\Psi}_N$, respectively, according to the size of the region $\Gamma_r$ (i.e., according to the radius $\rho$) can compensate this problem.
\item[(4)] Table \ref{tab:err4a} shows that the results for splines that use only ground data in $\Gamma_r$ are comparable to the results for the optimized kernels from Tables \ref{tab:err1} and \ref{tab:err2}, at least if $\eps_1=\eps_2$. In case of higher noise levels in $\Gamma_r$ than on $\Omega_R$ (i.e., $\eps_2>\eps_1$), it becomes clear that the combination of ground and satellite data by optimized kernels yields improved results. Only for a small spherical cap radius $\rho=0.02$ and a small $\eps_1=0.001$ (cf. Table \ref{tab:err2}), spherical splines behave significantly better than the optimized kernels. However, as already mentioned in (3), this can be overcome by increasing $N$ for the optimized kernels (noting that the splines that we use are non-band-limited).
\item[(5)] In Table \ref{tab:err4} we can see that for our test example pure downward continuation from satellite data on $\Omega_R$ yields significantly worse results than all other tested methods. This holds true especially for TSVD. RCM reveals a better behaviour but can compete with the other methods only for high noise levels and $\eps_2$  significantly larger than $\eps_1$ (e.g., $\eps_1=0.1$ and $\eps_2=5\eps_1$). Remembering Remark \ref{rem:lockernels}, one might consider using the symbols of the Tikhonov kernel for the functional $\mathcal{F}$ instead of the symbols of Shannon-type kernels to obtain optimized kernels $\Phi_N$ and $\tilde{\Psi}_N$ that further improve the results for the latter situation.
\end{itemize}

\begin{rem}
Condition \eqref{eqn:ass1} suggests that the parameter $\beta_N$ should be significantly smaller than $\alpha_{N,n}$ and $\tilde{\alpha}_{N,n}$. This does not seem to be reflected by the parameters in Tables \ref{tab:err1} and \ref{tab:err2}. But Condition \eqref{eqn:propx} implies a relation between $N$ and $\eps_1$, $\eps_2$ that, for the choice $N=100$, leads to significantly smaller $\eps_1$, $\eps_2$ than those chosen in Tables \ref{tab:err1} and \ref{tab:err2}. Indeed, if we choose, e.g., $\eps_1=\eps_2=10^{-10}$  and repeat the above examples for $\rho=0.5$, we obtain a set of good parameters $\beta_N=10^{-6}$, $\alpha_{N,n}=10^4$, $\tilde{\alpha}_{N,n}=10^4$ that reflects the expected behaviour. However, our goal was to test realistic noise levels for the intended applications and show that the proposed method via optimized kernels works well for these scenarios.
\end{rem}

\begin{table}
    (a) \,\,\begin{tabular}{ | c || c | c || c | c | }
   \hline
	$\varepsilon_1$ & \multicolumn{2}{|c||}{Shannon} &\multicolumn{2}{|c|}{Optimized}\\ \cline{2-5}
	$(\varepsilon_2=\varepsilon_1)$ &  Rel. Error & Cut-Off $M$ & Rel. Error & Param. $\beta_N,\tilde{\alpha}_{N,n},\alpha_{N,n}$\\\hline
  0.001& 0.008 & 0 & 0.003 & $3\cdot 10^3, 10^3, 10^3$\\
  0.01 & 0.012& 0 & 0.010 & $7\cdot 10^3, 10^3, 10^3$\\
	0.05 & 0.050 & 0 & 0.048 & $7\cdot 10^3, 10^3, 2\cdot10^2$ \\
	0.1 & 0.098 & 0 & 0.096& $7\cdot 10^3, 10^3,2\cdot10^2$\\
	\hline
\end{tabular}\\[2ex]
 (b) \,\,\begin{tabular}{ | c || c | c || c | c | }
   \hline
	$\varepsilon_1$ & \multicolumn{2}{|c||}{Shannon} &\multicolumn{2}{|c|}{Optimized}\\ \cline{2-5}
	$(\varepsilon_2=2\varepsilon_1)$ &  Rel. Error & Cut-Off $M$ & Rel. Error & Param. $\beta_N,\tilde{\alpha}_{N,n},\alpha_{N,n}$\\\hline
  0.001 & 0.008 & 0 & 0.003& $5\cdot 10^2, 10^3, 2\cdot10^2$\\
  0.01 & 0.021 & 0 & 0.019& $3\cdot 10^{-2}, 5\cdot 10^2, 5\cdot 10^2$\\
	0.05 & 0.096 & 28 & 0.096 & $3\cdot 10^{-2}, 10^3, 2\cdot 10^2$ \\
	0.1 & 0.188 & 28 & 0.188& $3\cdot 10^{-2}, 10^3, 2\cdot 10^2$\\
	\hline
\end{tabular}\\[2ex]
 (c) \,\,\begin{tabular}{ | c || c | c || c | c | }
   \hline
	$\varepsilon_1$ & \multicolumn{2}{|c||}{Shannon} &\multicolumn{2}{|c|}{Optimized}\\ \cline{2-5}
	$(\varepsilon_2=5\varepsilon_1)$ &  Rel. Error & Cut-Off $M$ & Rel. Error & Param. $\beta_N,\tilde{\alpha}_{N,n},\alpha_{N,n}$\\\hline
  0.001 & 0.009 & 0 & 0.005& $7\cdot 10^{-2}, 10^3, 10^3$\\
  0.01 & 0.046 & 36 & 0.041& $10^{-2}, 10^2, 10^2$\\
	0.05 & 0.213 & 36 & 0.203 &  $10^{-2}, 10^2, 10^2$ \\
	0.1 & 0.424 & 38 & 0.406 &  $10^1, 3\cdot 10^6, 6\cdot10^5$\\
	\hline
\end{tabular}

\caption{Relative errors $\|U_N^\varepsilon-U^+\|_{L^2(\tilde{\Gamma}_r)}/\|U^+\|_{L^2(\tilde{\Gamma}_r)}$ and $\|U_M^{\varepsilon,Sh}-U^+\|_{L^2(\tilde{\Gamma}_r)}/\|U^+\|_{L^2(\tilde{\Gamma}_r)}$  for spherical cap radius $\rho=0.5$. The noise level $\varepsilon_2$ in $\Gamma_r$ is varied among (a)--(c).}\label{tab:err1}
\end{table}

\begin{table}
   (a) \,\,\begin{tabular}{ | c || c | c || c | c | }
   \hline
	$\varepsilon_1$ & \multicolumn{2}{|c||}{Shannon} &\multicolumn{2}{|c|}{Optimized}\\ \cline{2-5}
	$(\varepsilon_2=\varepsilon_1)$ &  Rel. Error & Cut-Off $M$ & Rel. Error & Param. $\beta_N,\tilde{\alpha}_{N,n},\alpha_{N,n}$\\\hline
  0.001& 0.014 & 0 & 0.011 &  $3\cdot 10^3, 10^3, 10^3$\\
  0.01 & 0.017& 0 & 0.014 & $7\cdot 10^3, 10^3, 2\cdot 10^2$\\
	0.05 & 0.052 & 0 & 0.047 & $7\cdot 10^3, 10^3, 2\cdot 10^2$ \\
	0.1 & 0.099 & 0 & 0.096& $7\cdot 10^3, 10^3, 2\cdot 10^2$\\
	\hline
\end{tabular}\\[2ex]
 (b) \,\,\begin{tabular}{ | c || c | c || c | c | }
   \hline
	$\varepsilon_1$ & \multicolumn{2}{|c||}{Shannon} &\multicolumn{2}{|c|}{Optimized}\\ \cline{2-5}
	$(\varepsilon_2=2\varepsilon_1)$ &  Rel. Error & Cut-Off $M$ & Rel. Error & Param. $\beta_N,\tilde{\alpha}_{N,n},\alpha_{N,n}$\\\hline
  0.001 & 0.014 & 0 & 0.011& $3\cdot 10^3, 10^3, 10^3$\\
  0.01 & 0.025 & 0 & 0.020& $7\cdot 10^3, 5\cdot 10^2, 5\cdot 10^2$\\
	0.05 & 0.099 & 26 & 0.094 & $7\cdot 10^{-1}, 10^3, 2\cdot 10^2$ \\
	0.1 & 0.194 & 26 & 0.191& $7\cdot 10^{-1}, 10^3, 2\cdot 10^2$\\
	\hline
\end{tabular}\\[2ex]
 (c) \,\,\begin{tabular}{ | c || c | c || c | c | }
   \hline
	$\varepsilon_1$ & \multicolumn{2}{|c||}{Shannon} &\multicolumn{2}{|c|}{Optimized}\\ \cline{2-5}
	$(\varepsilon_2=5\varepsilon_1)$ &  Rel. Error & Cut-Off $M$ & Rel. Error & Param. $\beta_N,\tilde{\alpha}_{N,n},\alpha_{N,n}$\\\hline
  0.001 & 0.015 & 0 & 0.012& $3\cdot 10^3, 5\cdot10^3, 10^3$\\
  0.01 & 0.051 & 42 & 0.045& $3\cdot 10^{-2}, 5\cdot 10^2, 10^2$\\
	0.05 & 0.235 & 40 & 0.234 & $7\cdot 10^{-3}, 5\cdot 10^2, 10^2$ \\
	0.1 & 0.469 & 38 & 0.470 & $7\cdot 10^{-3}, 5\cdot 10^2, 10^2$\\
	\hline
\end{tabular}

\caption{Relative error $\|U_N^\varepsilon-U^+\|_{L^2(\tilde{\Gamma}_r)}/\|U^+\|_{L^2(\tilde{\Gamma}_r)}$ and $\|U_M^{\varepsilon,Sh}-U^+\|_{L^2(\tilde{\Gamma}_r)}/\|U^+\|_{L^2(\tilde{\Gamma}_r)}$ for spherical cap radius $\rho=0.02$.  The noise level $\varepsilon_2$ in $\Gamma_r$ is varied among (a)--(c).}\label{tab:err2}
\end{table}

\begin{table}
\begin{center}
\begin{tabular}{ | c || c | c | }
   \hline
	$\varepsilon_2$ &  \multicolumn{2}{|c|}{Spline} \\ \cline{2-3}
	 &   $\rho=0.5$ & $\rho=0.02$   \\ 
	\hline
  0.001 & {0.004} & 0.003  \\
  0.01 &  {0.009}& 0.010 \\
	0.05 &  {0.046} & 0.049  \\
	0.1 &  0.093 & 0.100 \\
	0.25 &  0.230 & 0.251 \\
	0.5 &  0.461 & 0.502 \\
	\hline
\end{tabular}
\end{center}
\caption{Spline interpolation/approximation: relative error $\|U^{\varepsilon,Spline}-U^+\|_{L^2(\tilde{\Gamma}_r)}/\|U^+\|_{L^2(\tilde{\Gamma}_r)}$ for varying radius $\rho\in\{0.02,0.5\}$.}\label{tab:err4a}
\end{table}

\begin{table}
\begin{center}
\begin{tabular}{ | c || c | c || c | c | }
   \hline
	$\varepsilon_1$ &  \multicolumn{2}{|c||}{TSVD} & \multicolumn{2}{|c|}{RCM} \\ \cline{2-5}
	 &   Rel. Error & Cut-Off M & Rel. Error &  Param. $\alpha$\\ 
	\hline
  0.001 & {0.229} & 90& 0.099 & $10^{-4}$ \\
  0.01 &  {0.629}& 72 & 0.178& $6\cdot 10^{-4}$ \\
	0.05 &  {0.770} & 56 & 0.228 & $3\cdot 10^{-3}$ \\
	0.1 &  0.817 & 50 & 0.244 & $6\cdot 10^{-3}$ \\
	\hline
\end{tabular}
\end{center}
\caption{Downward continuation: relative errors $\|U_M^{\varepsilon,TSVD}-U^+\|_{L^2(\tilde{\Gamma}_r)}/\|U^+\|_{L^2(\tilde{\Gamma}_r)}$ and $\|U^{\varepsilon,RCM}_\alpha-U^+\|_{L^2(\tilde{\Gamma}_r)}/\|U^+\|_{L^2(\tilde{\Gamma}_r)}$.}\label{tab:err4}
\end{table}

\begin{figure}
(a)\begin{center}
\scalebox{0.4}{\includegraphics{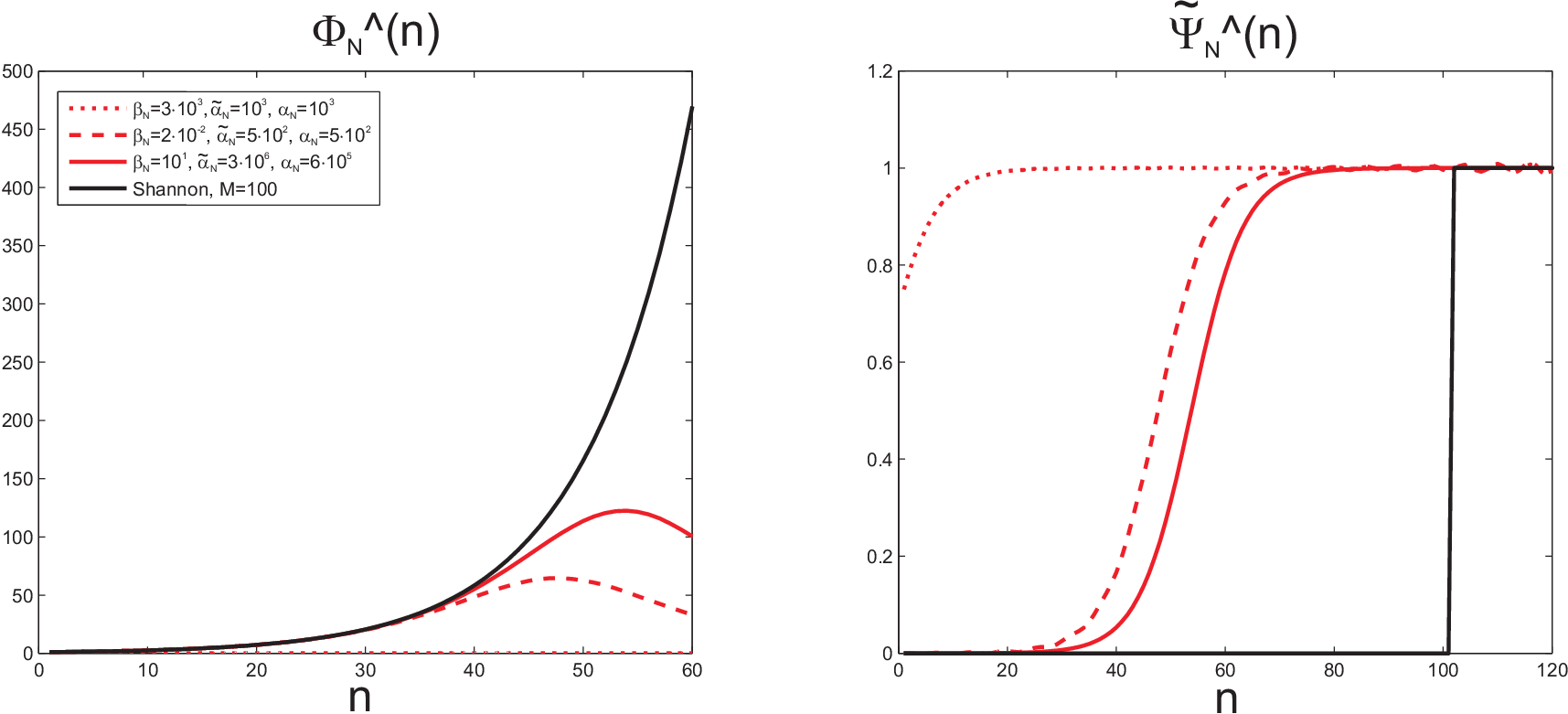}}\\[2ex]\end{center}
(b)\begin{center}\scalebox{0.4}{\includegraphics{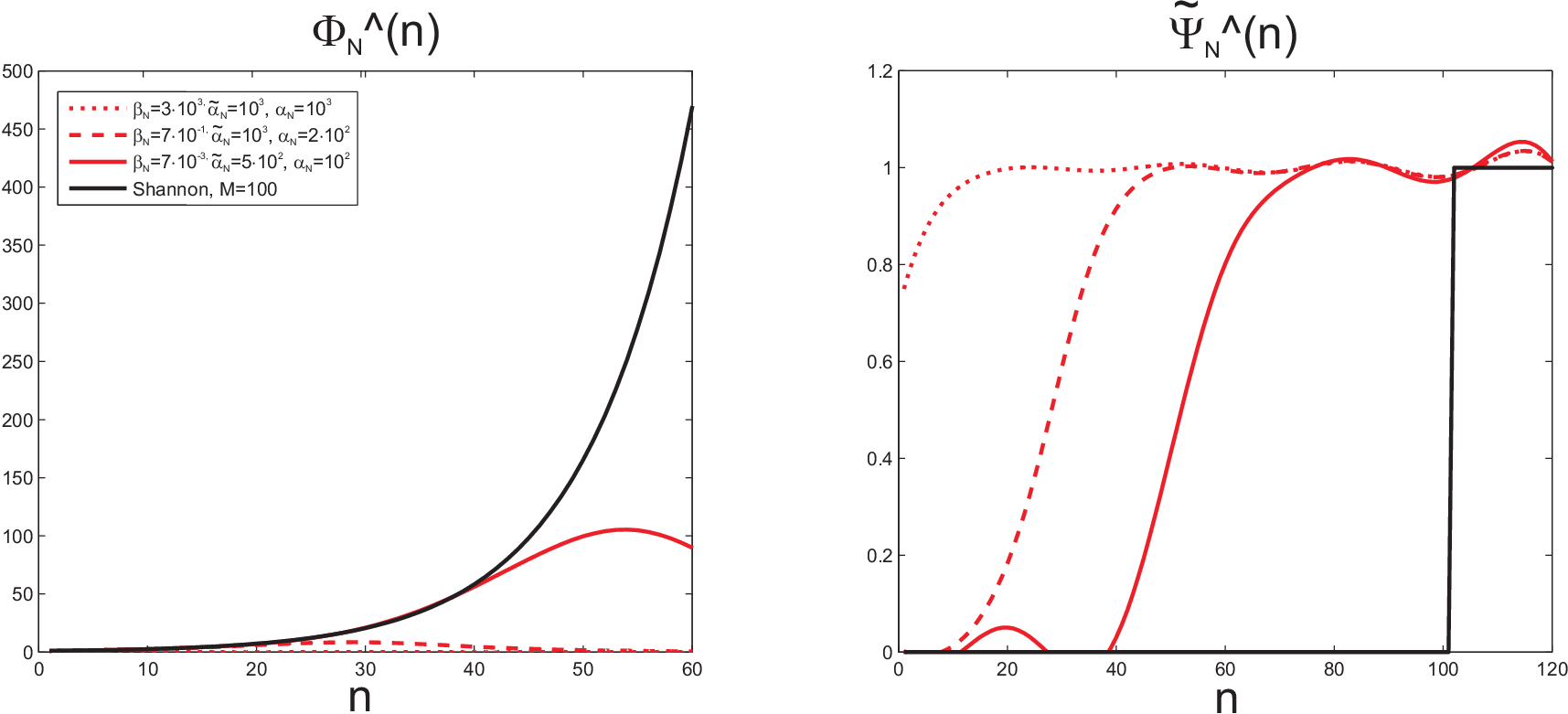}}\end{center}
\caption{Exemplary plots of the spectral behaviour of some of the optimized kernels used (a) in Table \ref{tab:err1} and (b) in Table \ref{tab:err2}.}\label{fig:kernelplot}
\end{figure}

\newpage

\section{The Vectorial Case}\label{sec:vect}

In some geophysical problems, especially in geomagnetism, it is not the scalar potential $U$ we are interested in but the vectorial gradient $\nabla U$. Just as well, it is often the gradient $\nabla U$ that is measured on $\Omega_R$ and $\Gamma_r$. Thus, we are not confronted with the scalar equations \eqref{eqn:11}--\eqref{eqn:13} but with the vectorial problem
\begin{align}
b=\nabla U,&\quad \textnormal{in }\Omega_r^{ext},\label{eqn:bvp441}
\\\Delta U=0,&\quad \textnormal{in }\Omega_r^{ext},\label{eqn:bvp442}
\\b=f_1,&\quad \textnormal{on }\Omega_R,\label{eqn:bvp443}
\\b=f_2,&\quad \textnormal{on }\Gamma_r.\label{eqn:bvp444}
\end{align}
Notationwise, we typically use lower-case letters to indicate vector fields, upper case letters for scalar fields, and boldface upper-case letters for tensor fields (the abbreviation $b$ for $\nabla U$ is simply chosen from common notation in geomagnetism). Starting from equations \eqref{eqn:bvp441}--\eqref{eqn:bvp444}, the general procedure for approximating $b^+=b|_{\Gamma_r}$  is essentially the same as for the scalar case treated in the previous sections. Therefore, we will be rather brief on the description and omit the proofs. An exception is given by Proposition \ref{prop:loc}, where we supply a vectorial counterpart to the localization property of Proposition \ref{prop:loc1}, and by Proposition \ref{lem:lockernel2} in order to indicate how the vectorial setting can be reduced to the previous scalar results.

Before dealing with the actual problem, it is necessary to introduce some basic vectorial framework. Here, we mainly follow the course of \cite{freeschrei09} and use the following set of vector spherical harmonics:
\begin{align}
{y}_{n,k}^{(1)}\left(\xi\right)&=\xi Y_{n,k}\left(\xi\right),\label{eqn:ynk1}
\\{y}_{n,k}^{(2)}\left(\xi\right)&=\frac{1}{\sqrt{n(n+1)}}\nabla^*_{\xi}Y_{n,k}\left(\xi\right),
\\{y}_{n,k}^{(3)}\left(\xi\right)&=\frac{1}{\sqrt{n(n+1)}}L_{\xi}^*Y_{n,k}\left(\xi\right),\label{eqn:ynk3}
\end{align}
for $n=1,2,\ldots$, and $k=1,\ldots, 2n+1$, with ${y}_{n,k}^{(1)}$ additionally being defined for $n=0$ and $k=1$. The operator $\nabla^*_\xi$ denotes the surface gradient, i.e., the tangential contribution of the gradient $\nabla_x$ (more precisely, $\nabla_x=\xi\frac{\partial}{\partial{r}}+\frac{1}{r}\nabla^*_{\xi}$, with $\xi=\frac{x}{|x|}$ and $r=|x|$). The surface curl gradient $L^*_{\xi}$ stands short for $\xi\wedge\nabla_{\xi}^*$ (with $\wedge$ being the vector product of two vectors $x,y\in\mathbb{R}^3$). Together, the functions \eqref{eqn:ynk1}--\eqref{eqn:ynk3} form an orthonormal basis of the space $l^2(\Omega)$ of vectorial functions that are square-integrable on the unit sphere. They are complemented by a set of tensorial Legendre polynomials $\mathbf{P}^{(i,i)}_n$ of degree $n$ and type $(i,i)$ that are defined via
\begin{align}
\mathbf{P}^{(1,1)}_n\left(\xi,\eta\right)&=\xi\otimes\eta \,P_n\left(\xi\cdot\eta\right),
\\\mathbf{P}^{(2,2)}_n\left(\xi,\eta\right)&=\frac{1}{{n(n+1)}}\nabla^*_{\xi}\otimes\nabla^*_{\eta}P_n\left(\xi\cdot\eta\right),
\\\mathbf{P}^{(3,3)}_n\left(\xi,\eta\right)&=\frac{1}{{n(n+1)}}L^*_{\xi}\otimes L^*_{\eta}P_n\left(\xi\cdot\eta\right).
\end{align}
The operator $\otimes$ denotes the tensor product $x\otimes y=xy^T$ of two vectors $x,y\in\mathbb{R}^3$. In analogy to the scalar case, vector spherical harmonics and tensorial Legendre polynomials are connected by an addition theorem. Since we are only dealing with vector fields of the form $\nabla U$ in this section, the vector spherical harmonics of type $i=3$ and the tensorial Legendre polynomials of type $(i,i)=(3,3)$ are not required and will be neglected for the remainder of this paper. The vectorial counterpart to the Sobolev space $\mathcal{H}_s(\Omega_r)$ is defined as
\begin{align}\label{eqn:sobspace2}
\mathfrak{h}_s(\Omega_r)=\Big\{f\in l^2(\Omega_r) :\|f\|_{\mathfrak{h}_s(\Omega_r)}^2=\sum_{i=1}^2\sum_{n=0_i}^\infty\sum_{k=1}^{2n+1}\left(n+\textnormal{\footnotesize$\frac{1}{2}$}\right)^{2s}\big|({f}_r^{(i)})^\wedge(n,k)\big|^2<\infty\Big\},
\end{align}
where $0_i=0$ if $i=1$ and $0_i=1$ if $i=2$, and $({f}_r^{(i)})^\wedge(n,k)$ denotes the Fourier coefficient of degree $n$, order $k$, and type $i$, i.e.,
\begin{align}
({f}_r^{(i)})^\wedge(n,k)=\int_{\Omega_r}f(y)\cdot \frac{1}{r}y_{n,k}^{(i)}\left(\frac{y}{|y|}\right)d\omega(y).
\end{align}
The space of band-limited tensorial kernels with maximal degree $N$ is defined as
\small\begin{align}\label{eqn:polnj2}
\mathbf{Pol}_{N}=\left\{\mathbf{K}(x,y)=\sum_{i=1}^2\sum_{n=0_i}^{N}\sum_{k=1}^{2n+1}\mathbf{K}^\wedge(n) {y}^{(i)}_{n,k}\left(\xi\right)\otimes{y}^{(i)}_{n,k}\left(\eta\right):\mathbf{K}^\wedge(n)\in\mathbb{R}\right\}.
\end{align}\normalsize
The kernels $\mathbf{Pol}_{N}$ are tensor-zonal. In particular, this means that the absolute value $|{\mathbf{K}}(x,y)|$ depends only on the scalar product $\xi\cdot\eta$ (this does not have to hold true for the kernel ${\mathbf{K}}(x,y)$ itself).

\begin{rem}
In geomagnetic modeling, another set of vector spherical harmonics is used more commonly than the one applied in this paper. We have used the basis \eqref{eqn:ynk1}--\eqref{eqn:ynk3} since it is generated by simpler differential operators, which reduces the effort to obtain a vectorial version of the localization principle later on. However, both basis systems eventually yield the same results. More information on the other basis system and its application in geomagnetism can be found, e.g., in \cite{backus96}, \cite{ger11}, \cite{ger12}, \cite{gubbins12}, and \cite{mayer06}.
\end{rem}

Similar to the scalar case in Subsection \ref{subsec:dc}, there is a vectorial upward continuation operator $t^{up}$ and a vectorial downward continuation operator $t^{down}$, defined via tensorial kernels with singular values $\sigma_n=\left(\frac{r}{R}\right)^{n+1}$ and $\frac{1}{\sigma_n}$, respectively (note that in the scalar case, $\sigma_n=\left(\frac{r}{R}\right)^{n}$). The downward continuation operator can be approximated by a bounded operator $t_N:l^2(\Omega_R)\to l^2(\Omega_r)$:
\begin{align}
{t_N[f_1](x)=\int_{\Omega_R}\mathbf{\Phi}_N(x,y) f_1(y)d\omega(y),\quad x\in\Omega_r,}
\end{align}
with
\begin{align}
{\mathbf{\Phi}_N(x,y)=\sum_{i=1}^2\sum_{n=0_i}^{N}\sum_{k=1}^{2n+1}\mathbf{\Phi}_{N}^\wedge(n)\frac{1}{r}{y}^{(i)}_{n,k}\left(\xi\right)\otimes \frac{1}{R}{y}^{(i)}_{n,k}\left(\eta\right).}\label{eqn:scalkernel2}
\end{align}
The symbols $\mathbf{\Phi}_{N}^\wedge(n)$ need to satisfy the same conditions (a), (b) as for the scalar analogue in Subsection \ref{subsec:dc}. A refinement by local data in $\Gamma_r$ is achieved by the vectorial wavelet operator $\tilde{{w}}_N:l^2({\Gamma}_r)\to l^2(\tilde{\Gamma}_r)$:
\begin{align}\label{eqn:wjdef}
{\tilde{{w}}_N[f_2](x)=\int_{\mathcal{C}_r(x,\rho)}\tilde{\mathbf{\Psi}}_N(x,y) f_2(y)d\omega(y),\quad x\in\tilde{\Gamma}_r,}
\end{align}
with
\begin{align}
{\tilde{\mathbf{\Psi}}_N(x,y)=\sum_{i=1}^2\sum_{n=0_i}^{\lfloor\kappa N\rfloor}\sum_{k=1}^{2n+1}\tilde{ \mathbf{\Psi}}_{N}^\wedge(n)\frac{1}{r}{y}^{(i)}_{n,k}\left(\xi\right)\otimes \frac{1}{r}{y}^{(i)}_{n,k}\left(\eta\right)},\label{eqn:wavkernel2}
\end{align}
for some fixed $\kappa>1$, and 
\begin{align}
{\tilde{ \mathbf{\Psi}}_{N}^\wedge(n)=\tilde{ \mathbf{\Phi}}_{N}^\wedge(n)-\mathbf{\Phi}_{N}^\wedge(n)\sigma_n.}
\end{align}
The symbols $\tilde{ \mathbf{\Phi}}_{N}^\wedge(n)$ are assumed to satisfy conditions $($a'$)$ and $($b'$)$ from Subsection \ref{subsec:comb}. As in the scalar case, the input data $f_1$, $f_2$ is perturbed by deterministic noise $e_1\in l^2(\Omega_R)$ and $e_2\in l^2(\Omega_r)$, so that we are dealing with $f_1^{\eps_1}=f_1+\eps_1 e_1$ and $f_2^{\eps_2}=f_2+\eps_2 e_2$. An approximation of $b^+$ in $\tilde{\Gamma}_r$ is then defined by
\begin{align}\label{eqn:approxbj3}
{b_N^{\eps}={t}_{N}[f_1^{\eps_1}]+\tilde{{w}}_{N}[f_2^{\eps_2}].}
\end{align}
A similar error estimate as in \eqref{eqn:errorestfinal} motivates the minimization of the functional
\begin{equation}\label{eqn:mineq4}
{\begin{aligned}
\mathcal{F}(\mathbf{\Phi}_N,\tilde{\mathbf{\Psi}}_N)=&\sum_{n=0}^{\lfloor\kappa N\rfloor}\tilde{\alpha}_{N,n}\big|1-\tilde{\mathbf{\Phi}}_{N}^\wedge(n)\big|^2+\sum_{n=0}^{N}\alpha_{N,n}\big|1-{\mathbf{\Phi}}_{N}^\wedge(n)\sigma_n\big|^2
\\&+\beta_N\sum_{n=0}^{N}\big|\mathbf{\Phi}_N^\wedge(n)\big|^2+8\pi^2 r^4\big\|\tilde{\mathbf{\Psi}}_N\big\|^2_{L^2([-1,1-\rho])}.
\end{aligned}}
\end{equation}
In the vectorial case, the term $\|\tilde{\mathbf{\Psi}}_N\|^2_{L^2([-1,1-\rho])}$ is not quite as basic as in the scalar case \eqref{eqn:1dint} but still expressible by elementary means. For more details on the vectorial and tensorial setup, the reader is referred to \cite{freeschrei09}. Now, we are all set to state the vectorial counterparts to the theoretical results from Section \ref{sec:theo}. As mentioned earlier, the proofs are mostly omitted due to their similarity. 

\begin{lem}\label{prop:minsol4}
Assume that all parameters $\tilde{\alpha}_{N,n}$, $\alpha_{N,n}$, and $\beta_N$ are positive. Then there exists unique minimizers $\mathbf{\Phi}_N\in\textnormal{\textbf{Pol}}_{N}$ and $\tilde{\mathbf{\Psi}}_N\in\textnormal{\textbf{Pol}}_{\lfloor\kappa N\rfloor}$ of the functional $\mathcal{F}$ in \eqref{eqn:mineq4} that are determined by $\phi=(\mathbf{\Phi}_N^\wedge(1)\sigma_1,\ldots,\mathbf{\Phi}_N^\wedge(N)\sigma_{N},\tilde{\mathbf{\Phi}}_{N}^\wedge(1),\ldots,\tilde{\mathbf{\Phi}}_{N}^\wedge(\lfloor\kappa N\rfloor))^T$ which solves the linear equations
\begin{align}\label{eqn:lineq4}
\mathbf{M}{\phi}=\alpha,
\end{align}
where 
\begin{align}
\mathbf{M}=\left(\begin{array}{c|c}\mathbf{D}_1+\mathbf{P}_1&-\mathbf{P}_2\\\hline-\mathbf{P}_3&\mathbf{D}_2+\mathbf{P}_4\end{array}\right),
\end{align}
and $\alpha=(\alpha_{N,1},\ldots,\alpha_{N,N},\tilde{\alpha}_{N,1},\ldots,\tilde{\alpha}_{N,\lfloor\kappa N\rfloor})^T$. The diagonal matrices $\mathbf{D}_1$, $\mathbf{D}_2$ are given by 
\begin{align}
\mathbf{D}_1=\textnormal{diag}\left(\frac{\beta_N}{\sigma_n^2}+\alpha_{N,n}\right)_{n=0,\ldots,N}, \qquad\mathbf{D}_2=\textnormal{diag}\big(\tilde{\alpha}_{N,n}\big)_{n=0,\ldots,\lfloor\kappa N\rfloor},
\end{align}
whereas $\mathbf{P}_1$,\ldots, $\mathbf{P}_4$ are submatrices of the Gram matrix $\big(P_{n,m}^\rho\big)_{n,m=0,\ldots, \lfloor\kappa N\rfloor}$. More precisely,
\begin{align}
\mathbf{P}_1=\big(P_{n,m}^\rho\big)_{n,m=0,\ldots,N},&\qquad \mathbf{P}_2=\big(P_{n,m}^\rho\big)_{n=0,\ldots,N\atop m=0,\ldots,N\lfloor\kappa N\rfloor},
\\\mathbf{P}_3=\big(P_{n,m}^\rho\big)_{n=0,\ldots,\lfloor\kappa N\rfloor\atop m=0,\ldots,N},&\qquad\mathbf{P}_4=\big(P_{n,m}^\rho\big)_{n,m=0,\ldots,\lfloor\kappa N\rfloor}.
\end{align}
with
\small\begin{align}
P_{n,m}^\rho=&\frac{(2n+1)(2m+1)}{16\pi^2}\int_{-1}^{1-\rho}P_n\left(t\right)P_m\left(t\right)dt
\\&+\frac{(2n+1)(2m+1)}{16\pi^2n(n+1)m(m+1)}\bigg(\int_{-1}^{1-\rho}(1+t^2)P_n'(t)P_m'(t)dt+\int_{-1}^{1-\rho}(1-t^2)^2P_n''(t)P_m''(t)dt\nonumber
\\&\quad-\int_{-1}^{1-\rho}t(1-t^2)P_n''(t)P_m'(t)dt-\int_{-1}^{1-\rho}t(1-t^2)P_n'(t)P_m''(t)dt\bigg),\nonumber
\end{align}\normalsize
for $n\not=0$ and $m\not=0$. If $n=0$ or $m=0$, then $P_{n,m}^\rho=\frac{(2n+1)(2m+1)}{16\pi^2}\int_{-1}^{1-\rho}P_n\left(t\right)P_m\left(t\right)dt$. By $P_n'$, $P_n''$ we mean the first and second order derivatives of the Legendre polynomials.
\end{lem}

To show that the left expression of \eqref{eqn:altterms} vanishes in the scalar case as $N\to\infty$, we used the localization property from Proposition \ref{prop:loc1}. A similar result is needed to prove Theorem \ref{thm:convres2}. The corresponding vectorial localization property for the Shannon-type kernel is stated in the next proposition.

\begin{prop}\label{prop:loc} 
If $f\in\mathfrak{h}_s(\Omega_r)$, $s\geq 2$, it holds that
\begin{align}\label{eqn:locpropvect}
\lim_{N\to\infty}\left\|\int_{\Omega_r\setminus\mathcal{C}_r(\rho,\cdot)}\tilde{\mathbf{\Psi}}^{Sh}_N(\cdot,y)f(y)d\omega(y)\right\|_{l^2 (\Omega_r)}=0,
\end{align}
where $\tilde{\mathbf{\Psi}}^{Sh}_N$ is the tensorial Shannon-type kernel with symbols $\big(\tilde{\mathbf{\Psi}}_N^{Sh}\big)^\wedge(n)=1$, if $N +1\leq n\leq \lfloor\kappa N\rfloor$, and $\big(\tilde{\mathbf{\Psi}}_N^{Sh}\big)^\wedge(n)=0$, else.
\end{prop}

\begin{proof}
We first observe that
\begin{align}\label{eqn:vectloceq1}
&\int\limits_{\Omega_r\setminus\mathcal{C}_r(x,\rho)}\tilde{\mathbf{\Psi}}^{Sh}_{N}(x,y) f(y)d\omega(y)
\\&=\int\limits_{\Omega_r\setminus\mathcal{C}_r(x,\rho)}\left(\sum_{i=1}^2\sum_{n=N+1}^{\lfloor\kappa N\rfloor}\sum_{k=1}^{2n+1}\frac{1}{r}y_{n,k}^{(i)}\left(\xi\right)\otimes \frac{1}{r}y_{n,k}^{(i)}\left(\eta\right)\right) f(r\eta)d\omega(r\eta)\nonumber
\\&=\sum_{i=1}^2\sum_{n=N+1}^{\lfloor\kappa N\rfloor}\sum_{k=1}^{2n+1}\frac{1}{r}y_{n,k}^{(i)}\left(\xi\right)\int\limits_{\Omega_r\setminus\mathcal{C}_r(x,\rho)}\frac{1}{r}y_{n,k}^{(i)}\left(\eta\right)\cdot f(r\eta)d\omega(r\eta),\nonumber
\end{align}
where, as always, $\xi=\frac{x}{|x|}$ and $\eta=\frac{y}{|y|}$. Since $f$ is of class $\mathfrak{h}_s(\Omega_r)$, $s\geq 2$, it follows that $f(x)=\xi F_1(x)+\nabla^*_\xi F_2(x)$ for some scalar functions $F_1$ of class $\mathcal{H}_s(\Omega_r)$ and $F_2$ of class $\mathcal{H}_{s+1}(\Omega_r)$. Taking a closer look at the terms of type $i=1$ in \eqref{eqn:vectloceq1}, using the orthogonality of $\xi$ and $\nabla^*_\xi$, we obtain
\begin{align}
&\sum_{k=1}^{2n+1}\frac{1}{r}y_{n,k}^{(1)}\left(\xi\right)\int\limits_{\Omega_r\setminus\mathcal{C}_r(x,\rho)}\frac{1}{r}y_{n,k}^{(1)}\left(\eta\right)\cdot f(r\eta)d\omega(r\eta)\label{eqn:eqn1}
\\&=\sum_{k=1}^{2n+1}\frac{1}{r}y_{n,k}^{(1)}\left(\xi\right)\int\limits_{\Omega_r\setminus\mathcal{C}_r(x,\rho)}\frac{1}{r}y_{n,k}^{(1)}\left(\eta\right)\cdot \left(\eta F_1(r\eta)+\nabla^*_{\eta}F_2(r\eta)\right)d\omega(r\eta)\nonumber
\\&=\sum_{k=1}^{2n+1}\frac{1}{r^2}\xi Y_{n,k}\left(\xi\right)\int\limits_{\Omega_r\setminus\mathcal{C}_r(x,\rho)}Y_{n,k}\left(\eta\right)F_1(r\eta)d\omega(r\eta)\nonumber
\\&=\frac{1}{r^2}\xi\int\limits_{\Omega_r\setminus\mathcal{C}_r(x,\rho)}\frac{2n+1}{4\pi}P_n\left(\xi\cdot\eta\right)F_1(r\eta)d\omega(r\eta).\nonumber
\end{align}
For the terms of type $i=2$, the use of Green's formulas (cf. \cite{ger11} for more details) and the addition theorem for spherical harmonics implies
\begin{align}
&\xi\wedge\sum_{k=1}^{2n+1}\frac{1}{r}y_{n,k}^{(2)}\left(\xi\right)\int\limits_{\Omega_r\setminus\mathcal{C}_r(x,\rho)}\frac{1}{r}y_{n,k}^{(2)}\left(\eta\right)\cdot f(r\eta)d\omega(r\eta)\label{eqn:shit}
\\&=\xi\wedge\sum_{k=1}^{2n+1}\frac{1}{r}y_{n,k}^{(2)}\left(\xi\right)\nonumber
\int\limits_{\Omega_r\setminus\mathcal{C}_r(x,\rho)}\frac{1}{r}\frac{1}{\sqrt{n(n+1)}}\nabla^*_{\eta}Y_{n,k}\left(\eta\right)\cdot \nabla^*_{\eta}F_2(r\eta)d\omega(r\eta)\nonumber
\\&=-\xi\wedge\sum_{k=1}^{2n+1}\frac{1}{r^2}\frac{1}{{n(n+1)}}\nabla^*_{\xi}Y_{n,k}\left(\xi\right)\int\limits_{\Omega_r\setminus\mathcal{C}_r(x,\rho)}\left(\Delta^*_{\eta}Y_{n,k}\left(\eta\right)\right)F_2(r\eta)d\omega(r\eta)\nonumber
\\&\quad+\xi\wedge\sum_{k=1}^{2n+1}\frac{1}{r^2}\frac{1}{{n(n+1)}}\nabla^*_{\xi}Y_{n,k}\left(\xi\right)\int_{\partial\mathcal{C}_r(x,\rho)}F_2(r\eta)\frac{\partial}{\partial\nu_\eta}Y_{n,k}\left(\eta\right)d\sigma(r\eta)\nonumber
%
\\&=\sum_{k=1}^{2n+1}\frac{1}{r^2}L^*_{\xi}Y_{n,k}\left(\xi\right)\int\limits_{\Omega_r\setminus\mathcal{C}_r(x,\rho)}Y_{n,k}\left(\eta\right)F_2(r\eta)d\omega(r\eta)\nonumber
\\&\quad+\frac{1}{r^2}\frac{1}{{ n(n+1)}}\int_{\partial\mathcal{C}_r(x,\rho)}F_2(r\eta)\frac{2n+1}{{4\pi}}L^*_{\xi}\frac{\partial}{\partial\nu_\eta}P_{n}\left(\xi\cdot\eta\right)d\sigma(r\eta)\nonumber
\\&=-\frac{1}{r^2}\int\limits_{\Omega_r\setminus\mathcal{C}_r(x,\rho)}F_2(r\eta)\frac{2n+1}{4\pi}L_{\eta}^*P_n\left(\xi\cdot\eta\right)d\omega(r\eta)\nonumber
\\&\quad+\frac{1}{r^2}\frac{1}{{ n(n+1)}}\int_{\partial\mathcal{C}_r(x,\rho)}F_2(r\eta)\frac{2n+1}{{4\pi}}L^*_{\xi}\frac{\partial}{\partial\nu_\eta}P_{n}\left(\xi\cdot\eta\right)d\sigma(r\eta)\nonumber
\\&=\frac{1}{r^2}\int\limits_{\Omega_r\setminus\mathcal{C}_r(x,\rho)}\frac{2n+1}{4\pi}P_n\left(\xi\cdot\eta\right)L_{\eta}^*F_2(r\eta)d\omega(r\eta)\nonumber
\\&\quad+\frac{1}{r^2}\int_{\partial\mathcal{C}_r(x,\rho)}\tau_{\eta}F_2(r\eta)\frac{2n+1}{{4\pi}}P_{n}\left(\xi\cdot\eta\right)d\sigma(r\eta)\nonumber
\\&\quad+\frac{1}{r^2}\frac{1}{{ n(n+1)}}\int_{\partial\mathcal{C}_r(x,\rho)}F_2(r\eta)\frac{2n+1}{{4\pi}}L^*_{\xi}\frac{\partial}{\partial\nu_\eta}P_{n}\left(\xi\cdot\eta\right)d\sigma(r\eta)\nonumber
\end{align}\normalsize
By $\frac{\partial}{\partial\nu_\eta}$ we mean the normal derivative at $r\eta\in\partial\mathcal{C}_r(x,\rho)$, and $\tau_\eta$ denotes the tangential unit vector at $r\eta\in\partial\mathcal{C}_r(x,\rho)$. The reason for the application of $\xi\wedge$ in \eqref{eqn:shit} is that we can then work with the operator $L_{\xi}^*$ instead of $\nabla_\xi ^*$. The surface curl gradient has the nice property $L_{\xi}^*P_{n}\left(\xi\cdot\eta\right)=-L_{\eta}^*P_{n}\left(\xi\cdot\eta\right)$ which we have used in the seventh line of Equation \eqref{eqn:shit}. Furthermore, since $\xi$ and $\nabla_{\xi}^*$ are orthogonal, the convergence of \eqref{eqn:shit} for $N\to \infty$ also implies convergence for the same expression without the application of  $\xi\wedge$. 
Now, we can use the scalar localization result from Proposition \ref{prop:loc} to obtain 
\begin{align}
\lim_{N\to\infty}\Bigg\|\frac{\cdot}{|\cdot|}\int\limits_{\Omega_r\setminus\mathcal{C}_r(\rho,\cdot)}\Bigg(\underbrace{\sum_{n=N+1}^{\lfloor\kappa N\rfloor}\frac{2n+1}{4\pi r^2}P_n\left(\frac{\cdot}{|\cdot|}\cdot\eta\right)}_{=\tilde{\Psi}_N^{Sh}(\cdot,y)}\Bigg)F_1(r\eta)d\omega(r\eta)\Bigg\|_{l^2(\Omega_r)}=0
\end{align}
and
\begin{align}
\lim_{N\to\infty}\Bigg\|\int\limits_{\Omega_r\setminus\mathcal{C}_r(\rho,\cdot)}\left(\sum_{n=N+1}^{\lfloor\kappa N\rfloor}\frac{2n+1}{4\pi r^2}P_n\left(\frac{\cdot}{|\cdot|}\cdot\eta\right) \right)L_{\eta}^*F_2(r\eta)d\omega(r\eta)\Bigg\|_{l^2(\Omega_r)}=0,
\end{align}
which deals with the relevant contributions to the asymptotic behaviour of \eqref{eqn:eqn1} and \eqref{eqn:shit}. It remains to investigate the boundary integrals appearing on the right hand side of \eqref{eqn:shit}. Observing the differential equation $(1-t^2)P_n''(t)-2tP_n'(t)+n(n+1)P_n(t)=0$, $t\in[-1,1]$, for Legendre polynomials leads to
\begin{align}\label{eqn:bvintzero}
&\!\!\!\!\!\!\frac{1}{r^2}\int_{\partial\mathcal{C}_r(x,\rho)}\tau_{\eta}F_2(r\eta)\frac{2n+1}{{4\pi}}P_{n}\left(\xi\cdot\eta\right)d\sigma(r\eta)
\\&\!\!\!\!\!\!+\frac{1}{r^2}\frac{1}{{ n(n+1)}}\int_{\partial\mathcal{C}_r(x,\rho)}F_2(r\eta)\frac{2n+1}{{4\pi}}L^*_{\xi}\frac{\partial}{\partial\nu_\eta}P_{n}\left(\xi\cdot\eta\right)d\sigma(r\eta)\nonumber
\\=&\frac{1}{r^2}\int_{\partial\mathcal{C}_r(x,\rho)}\tau_{\eta}F_2(r\eta)\frac{2n+1}{{4\pi}}\bigg(P_{n}(\xi\cdot\eta)+\frac{1}{ n(n+1)}\big(1-(\xi\cdot\eta)^2\big)P_{n}''(\xi\cdot\eta)\nonumber
\\&\qquad\qquad\qquad\qquad\qquad\qquad-\frac{2}{ n(n+1)}(\xi\cdot\eta)P_{n}'(\xi\cdot\eta)\bigg)d\sigma(r\eta)\nonumber
\\=&0.\nonumber
\end{align}
Combining \eqref{eqn:vectloceq1}--\eqref{eqn:bvintzero} implies the desired property \eqref{eqn:locpropvect}.
\end{proof}

\begin{thm}\label{thm:convres2}
Assume that parameters $\alpha_{N,n}$, $\tilde{\alpha}_{N,n}$, and $\beta_N$ are positive and suppose that, for some $\delta>0$ and $\kappa>1$,
\begin{align}
\inf_{n=0,\ldots, \lfloor \kappa N\rfloor}\frac{\beta_N\frac{1-\left(\frac{R}{r}\right)^{2(N+1)}}{1-\left(\frac{R}{r}\right)^2}+(\lfloor \kappa N\rfloor+1)^2}{{\alpha}_{N,n}}&=\mathcal{O}\left(N^{-2(1+\delta)}\right),\quad \textit{ for }N\to\infty.\label{eqn:ass4}
\end{align}
An analogous relation shall hold true for $\tilde{\alpha}_{N,n}$. Additionally, let every $N$ be associated with an $\varepsilon_1=\varepsilon_1(N)>0$ and an $\varepsilon_2=\varepsilon_2(N)>0$ such that 
\begin{align}
\lim_{N\to\infty}\varepsilon_1 \left(\frac{R}{r}\right)^{N}=\lim_{N\to\infty}\varepsilon_2\, {N}=0.\label{eqn:propx3}
\end{align}
The functions $f_1:\Omega_R\to\mathbb{R}^3$ and $f_2:\Gamma_r\to\mathbb{R}^3$, $r<R$, are supposed to be such that a unique solution $b$ of \eqref{eqn:bvp441}--\eqref{eqn:bvp444} exists and that the restriction $b^+$ is of class $\mathfrak{h}_s(\Omega_r)$, for some fixed $s\geq 2$. The erroneous input data is given by $f_1^{\eps_1}=f_1+\eps_1 e_1$ and $f_2^{\eps_2}=f_2+\eps_1 e_2$, with $e_1\in l^2(\Omega_R)$ and $e_2\in l^2(\Omega_r)$. If the kernels $\mathbf{\Phi}_N\in\textnormal{\textbf{Pol}}_{N}$ and the wavelet kernel $\tilde{\mathbf{\Psi}}_N\in\textnormal{\textbf{Pol}}_{\lfloor \kappa N\rfloor}$ are the minimizers of the functional $\mathcal{F}$ from \eqref{eqn:mineq4} and if $b_N^\eps$ is given as in \eqref{eqn:approxbj3}, then
\begin{align}
\lim_{N\to\infty}\left\|b^+-b_ {N}^{\varepsilon}\right\|_{l^2(\tilde{\Gamma}_r)}=0.\label{eqn:convres3}
\end{align}
\end{thm}

\begin{lem}\label{lem:lockernel2}
Assume that the parameters $\alpha_{N,n}$, $\tilde{\alpha}_{N,n}$, and $\beta_N$ are positive and suppose that, for some $\delta>0$ and $\kappa>1$,
\begin{align}
\inf_{n=0,\ldots, \lfloor \kappa N\rfloor}\frac{\beta_N\frac{1-\left(\frac{R}{r}\right)^{2(N+1)}}{1-\left(\frac{R}{r}\right)^2}+(\lfloor \kappa N\rfloor+1)^2}{{\alpha}_{N,n}}&=\mathcal{O}\left(N^{-2(1+\delta)}\right),\quad \textit{ for }N\to\infty.\label{eqn:ass5}
\end{align}
An analogous relation shall hold true for $\tilde{\alpha}_{N,n}$. If the scaling kernel $\mathbf{\Phi}_N\in\textnormal{\textbf{Pol}}_{N}$ and the wavelet kernel $\tilde{\mathbf{\Psi}}_N\in\textnormal{\textbf{Pol}}_{\lfloor \kappa N\rfloor}$ are the minimizers of the functional $\mathcal{F}$ from \eqref{eqn:mineq4}, then
\begin{align} \label{eqn:optloc3}
\lim_{N\to\infty}\frac{\big\|\tilde{\mathbf{\Psi}}_N\big\|^2_{L^2([-1,1-\rho])}}{\big\|\tilde{\mathbf{\Psi}}_N\big\|^2_{L^2([-1,1])}}=0.
\end{align}
\end{lem}

\begin{proof}
Set $F_N(t)={|\tilde{\mathbf{\Psi}}_N(x,y)|^2}/{\|\tilde{\mathbf{\Psi}}_N\|^2_{L^2([-1,1])}}$, where $t=\xi \cdot\eta $. Since $\tilde{\mathbf{\Psi}}_N$ is tensor-zonal, $F_N$ is well-defined and can be regarded as a density function of a random variable $t\in[-1,1]$. From here on, the proof is essentially the same as for Lemma \ref{lem:loc} and we have to show $\lim_{N\to\infty}E_N(t)=1$ and $\lim_{N\to\infty}V_N(t)=0$. We just indicate the proof for $E_N(t)$, the case of $V_N(t)$ follows analogously. Setting $\tilde{\Psi}_N(t)=|\tilde{\mathbf{\Psi}}_N(x,y)|$, again with  $t=\xi \cdot\eta $, we get
\begin{align}\label{eqn:ejtvect}
&2\pi r^2\int_{-1}^1t|\tilde{\Psi}_N(t)|^2dt=\int_{\Omega_r}\xi \cdot\eta \left|\tilde{\mathbf{\Psi}}_N\left(r\xi ,r\eta \right)\right|^2d\omega(r\eta)
\\&=\frac{1}{r^2}\sum_{i=1}^2\sum_{n=0_i}^{\lfloor \kappa N\rfloor}\sum_{k=1}^{2n+1}\sum_{m=0_i}^{{\lfloor \kappa N\rfloor}}\sum_{l=1}^{2m+1}\tilde{\mathbf{\Psi}}_N^\wedge(n)\tilde{\mathbf{\Psi}}_N^\wedge(m)y_{n,k}^{(i)}\left(\xi \right)\cdot y_{m,l}^{(i)}\left(\xi \right)\nonumber
\\&\qquad\qquad\qquad\qquad\qquad\qquad\quad\times\int_{\Omega_r}(\xi \cdot\eta)\,  y_{n,k}^{(i)}\left(\eta \right)\cdot y_{m,l}^{(i)}\left(\eta \right)d\omega(r\eta)\nonumber
\\&=\frac{1}{r^2}\sum_{n=0}^{\lfloor \kappa N\rfloor}\sum_{k=1}^{2n+1}\sum_{m=0}^{{\lfloor \kappa N\rfloor}}\sum_{l=1}^{2m+1}\tilde{\mathbf{\Psi}}_N^\wedge(n)\tilde{\mathbf{\Psi}}_N^\wedge(m)Y_{n,k}\left(\xi \right)Y_{m,l}\left(\xi \right)\nonumber
\\&\qquad\qquad\qquad\qquad\qquad\quad\times\int_{\Omega_r}(\xi \cdot\eta)\, Y_{n,k}\left(\eta \right)Y_{m,l}\left(\eta \right)d\omega(r\eta)\nonumber
\\&\quad+\frac{1}{r^2}\sum_{n=1}^{\lfloor \kappa N\rfloor}\sum_{k=1}^{2n+1}\sum_{m=1}^{{\lfloor \kappa N\rfloor}}\sum_{l=1}^{2m+1}\frac{\tilde{\mathbf{\Psi}}_N^\wedge(n)\tilde{\mathbf{\Psi}}_N^\wedge(m)}{n(n+1)m(m+1)}\nabla^*_{\xi }Y_{n,k}\left(\xi \right)\cdot\nabla^*_{\xi }Y_{m,l}\left(\xi \right)\nonumber
\\&\qquad\qquad\qquad\qquad\qquad\quad\times\xi \cdot\int_{\Omega_r}\eta \left(\nabla^*_{\eta }Y_{n,k}\left(\eta \right)\cdot\nabla^*_{\eta }Y_{m,l}\left(\eta \right)\right)d\omega(r\eta)\nonumber
\end{align}\normalsize
The first term on the right hand side of \eqref{eqn:ejtvect} can be written as a one-dimensional integral
\begin{align}\label{eqn:firstsum}
&\frac{1}{r^2}\sum_{n=0}^{\lfloor \kappa N\rfloor}\sum_{k=1}^{2n+1}\sum_{m=0}^{{\lfloor \kappa N\rfloor}}\sum_{l=1}^{2m+1}\tilde{\mathbf{\Psi}}_N^\wedge(n)\tilde{\mathbf{\Psi}}_N^\wedge(m)Y_{n,k}\left(\xi \right)Y_{m,l}\left(\xi \right)
\\&\qquad\qquad\qquad\qquad\quad\,\,\,\times\int_{\Omega_r}(\xi \cdot\eta)\, Y_{n,k}\left(\eta \right)Y_{m,l}\left(\eta \right)d\omega(r\eta)\nonumber
\\&=\sum_{n=0}^{\lfloor \kappa N\rfloor}\sum_{m=0}^{{\lfloor \kappa N\rfloor}}\frac{(2n+1)(2m+1)}{8\pi}\tilde{\mathbf{\Psi}}_N^\wedge(n)\tilde{\mathbf{\Psi}}_N^\wedge(m)\int_{-1}^1tP_{n}\left(t\right)P_{m}\left(t\right)dt.\nonumber
\end{align}
The second term requires significantly more effort. We start by observing that
\begin{align}
&\int_{\Omega_r}\eta \left(\nabla^*_{\eta }Y_{n,k}\left(\eta \right)\cdot\nabla^*_{\eta }Y_{m,l}\left(\eta \right)\right)d\omega(r\eta)\label{eqn:1green}
\\&=-\frac{1}{2}\int_{\Omega_r}Y_{m,l}\left(\eta \right)\nabla_{\eta }^*\cdot\left(\eta \otimes\nabla^*_{\eta }Y_{n,k}\left(\eta \right)\right)+Y_{n,k}\left(\eta \right)\nabla_{\eta }^*\cdot\left(\eta \otimes\nabla^*_{\eta }Y_{m,l}\left(\eta \right)\right)d\omega(r\eta)\nonumber
\end{align}

\begin{align}
&=-\frac{1}{2}\int_{\Omega_r}\nabla_{\eta }^*\left(Y_{n,k}\left(\eta \right)Y_{m,l}\left(\eta \right)\right)d\omega(r\eta)\nonumber
+\frac{n(n+1)}{2}\int_{\Omega_r}\eta  Y_{n,k}\left(\eta \right)Y_{m,l}\left(\eta \right)d\omega(r\eta)\nonumber
\\&\quad+\frac{m(m+1)}{2}\int_{\Omega_r}\eta  Y_{n,k}\left(\eta \right)Y_{m,l}\left(\eta \right)d\omega(r\eta)\nonumber
\\&=\frac{n(n+1)}{2}\int_{\Omega_r}\eta  Y_{n,k}\left(\eta \right)Y_{m,l}\left(\eta \right)d\omega(r\eta)\nonumber
+\frac{m(m+1)}{2}\int_{\Omega_r}\eta  Y_{n,k}\left(\eta \right)Y_{m,l}\left(\eta \right)d\omega(r\eta),\nonumber
\end{align}\normalsize
where we have used $\nabla_{\eta}^*\cdot\left(\eta\otimes\nabla^*_{\eta}Y_{n,k}\left(\eta\right)\right)=\eta\Delta_\eta^*Y_{n,k}\left(\eta\right)+\nabla_\eta^*Y_{n,k}\left(\eta\right)$. Plugging  \eqref{eqn:1green} into the second term on the right hand side of Equation \eqref{eqn:ejtvect}, and using the expression $2\nabla_\xi ^*Y_{n,k}(\xi )\cdot \nabla_\xi ^*Y_{m,l}(\xi )=\Delta_\xi ^*(Y_{n,k}(\xi )Y_{m,l}(\xi ))+(n(n+1)+m(m+1))Y_{n,k}(\xi )Y_{m,l}(\xi )$, leads to
\begin{align}
&\frac{1}{r^2}\sum_{n=1}^{\lfloor \kappa N\rfloor}\sum_{k=1}^{2n+1}\sum_{m=1}^{{\lfloor \kappa N\rfloor}}\sum_{l=1}^{2m+1}\frac{\tilde{\mathbf{\Psi}}_N^\wedge(n)\tilde{\mathbf{\Psi}}_N^\wedge(m)}{n(n+1)m(m+1)}\nabla^*_{\xi }Y_{n,k}\left(\xi \right)\cdot\nabla^*_{\xi }Y_{m,l}\left(\xi \right)\label{eqn:2term}
\\&\qquad\qquad\qquad\qquad\quad\times\xi \cdot\int_{\Omega_r}\eta \nabla^*_{\eta }Y_{n,k}\left(\eta \right)\cdot\nabla^*_{\eta }Y_{m,l}\left(\eta \right)d\omega(r\eta)\nonumber
\\&=\frac{1}{2r^2}\sum_{n=1}^{\lfloor \kappa N\rfloor}\sum_{k=1}^{2n+1}\sum_{m=1}^{{\lfloor \kappa N\rfloor}}\sum_{l=1}^{2m+1}\tilde{\mathbf{\Psi}}_N^\wedge(n)\tilde{\mathbf{\Psi}}_N^\wedge(m)\frac{n(n+1)}{m(m+1)}\nonumber
\\&\qquad\qquad\qquad\quad\qquad\qquad\times\int_{\Omega_r}(\xi\cdot\eta)\, Y_{n,k}\left(\xi \right)Y_{m,l}\left(\xi \right)Y_{n,k}\left(\eta \right)Y_{m,l}\left(\eta \right)d\omega(r\eta)\nonumber
\\&\quad+\frac{1}{2r^2}\sum_{n=1}^{\lfloor \kappa N\rfloor}\sum_{k=1}^{2n+1}\sum_{m=1}^{{\lfloor \kappa N\rfloor}}\sum_{l=1}^{2m+1}\tilde{\mathbf{\Psi}}_N^\wedge(n)\tilde{\mathbf{\Psi}}_N^\wedge(m)\nonumber
\\&\qquad\qquad\qquad\quad\qquad\qquad\times\int_{\Omega_r}(\xi\cdot\eta)\, Y_{n,k}\left(\xi \right)Y_{m,l}\left(\xi \right)Y_{n,k}\left(\eta \right)Y_{m,l}\left(\eta \right)d\omega(r\eta)\nonumber
\\&\quad+\frac{1}{2r^2}\sum_{n=1}^{\lfloor \kappa N\rfloor}\sum_{k=1}^{2n+1}\sum_{m=1}^{{\lfloor \kappa N\rfloor}}\sum_{l=1}^{2m+1}\frac{\tilde{\mathbf{\Psi}}_N^\wedge(n)\tilde{\mathbf{\Psi}}_N^\wedge(m)}{n(n+1)}\nonumber
\\&\qquad\qquad\qquad\quad\qquad\qquad\times\xi \cdot\Delta_{\xi }^*\int_{\Omega_r}\eta Y_{n,k}\left(\xi \right)Y_{m,l}\left(\xi \right)Y_{n,k}\left(\eta \right)Y_{m,l}\left(\eta \right)d\omega(r\eta)\nonumber
\\&=\sum_{n=1}^{\lfloor \kappa N\rfloor}\sum_{m=1}^{{\lfloor \kappa N\rfloor}}\tilde{\mathbf{\Psi}}_N^\wedge(n)\tilde{\mathbf{\Psi}}_N^\wedge(m)\left(\frac{1}{2}+\frac{n(n+1)-2}{2m(m+1)}\right)\frac{(2n+1)(2m+1)}{8\pi}\int_{-1}^1tP_n(t)P_m(t)dt,\nonumber
\end{align}\normalsize
where we have used the addition theorem, the property $\Delta_{\xi }^*P_n(\xi \cdot\eta)=\Delta_{\eta}^*P_n(\xi \cdot\eta)$, Green's formulas, and $\Delta_{\eta}^*\,\eta=-2\eta$ in the last row. Eventually, combining \eqref{eqn:ejtvect}, \eqref{eqn:firstsum}, and \eqref{eqn:2term}, we are lead to
\begin{align}\label{eqn:inttpsiN}
&\int_{-1}^1t|\tilde{\Psi}_N(t)|^2dt
\\&=\sum_{n=1}^{\lfloor \kappa N\rfloor}\sum_{m=1}^{{\lfloor \kappa N\rfloor}}\tilde{\mathbf{\Psi}}_N^\wedge(n)\tilde{\mathbf{\Psi}}_N^\wedge(m)\left(\frac{3}{4}+\frac{n(n+1)-2}{4m(m+1)}\right)\frac{(2n+1)(2m+1)}{8\pi^2r^2}\int_{-1}^1tP_n(t)P_m(t)dt.\nonumber
\end{align}\normalsize
Observing $\lim_{m,n\to\infty}\frac{3}{4}+\frac{n(n+1)-2}{4m(m+1)}=1$ if $m\in\{n-1,n,n+1\}$, we can proceed with \eqref{eqn:inttpsiN} in a similar manner as in  \eqref{eqn:et1} and \eqref{eqn:et11}. Together with
\begin{align}\label{eqn:etvect}
\|\tilde{\mathbf{\Psi}}_N\|_{L^2([-1,1])}^2=\sum_{n=0}^{\lfloor \kappa N\rfloor} \frac{2n+1}{4\pi^2r^2}|\tilde{\mathbf{\Psi}}_N^\wedge(n)|^2=\mathcal{O}(N^{-2\delta})+\sum_{n=N+1}^{\lfloor \kappa N\rfloor} \frac{2n+1}{4\pi^2r^2}|\tilde{\mathbf{\Psi}}_N^\wedge(n)|^2,
\end{align}
this leads to the desired property
\begin{align}
\lim_{N\to\infty}E_N(t)=\lim_{N\to\infty}\int_{-1}^1 t|F_N(t)|^2 dt=\lim_{N\to\infty}\frac{\int_{-1}^1 t|\tilde{\Psi}_N(t)|^2 dt}{\|\tilde{\mathbf{\Psi}}_N\|_{L^2([-1,1])}^2}=1,
\end{align}
concluding the proof.
\end{proof}

\section{Conclusion}\label{sec:concl}
The combination of satellite and ground data is an important step to obtain high-resolution models, e.g., of the Earth's gravitational and crustal magnetic field. The numerical examples in this article have shown that the proposed approach via optimized kernels yields improved results over a naive approximation by Shannon-type kernels. Furthermore, it proved superior to spline approximation/interpolation from ground data in $\Gamma_r$ and downward continuation of satellite data on $\Omega_R$ via TSVD/RCM over a wide range of scenarios. Splines can only compete if the noise level in $\Gamma_r$ is not larger than the noise level on $\Omega_R$. Downward continuation, on the other hand, is comparable only if the noise level on $\Omega_R$ is significantly smaller than in $\Gamma_r$. The combined approach presented here automatically weighs the different properties of ground and satellite data against each other. In this sense, a crucial ingredient for the construction of the optimized kernels is the choice of the parameters $\beta_N$, $\alpha_{N,n}$, $\tilde{\alpha}_{N,n}$, and the truncation degree $N$. They may be motivated by a priori knowledge of the noise levels $\eps_1$, $\eps_2$, and the size of the region $\Gamma_r$, but in general they will require an adequate yet to investigate parameter choice strategy.
\\[2ex]

\textbf{Acknowledgements.} This work was partly conducted at UNSW Australia and supported by a fellowship within the Postdoc-program of the German Academic Exchange Service (DAAD). The author thanks Ian Sloan, Rob Womersley, and Yu  Guang Wang for valuable discussions on Proposition \ref{prop:loc1} as well as the two reviewers for their valuable comments and suggestions to improve the paper.


\end{document}